\pdfoutput=1
\RequirePackage{ifpdf}
\ifpdf 
\documentclass[pdftex]{sigma}
\else
\documentclass{sigma}
\fi

\def\op{\mathrm{op}}
\def\genfd{\mathbf{k}}
\def\gg{\mathfrak{g}}

\def\CC{\mathcal{C}}
\def\OO{\mathcal{O}}

\def\Derc{\operatorname{Derc}}

\def\End{\operatorname{End}}
\def\Ker{\operatorname{ker}}
\def\hx{\hat{x}}
\def\hy{\hat{y}}
\def\ad{\operatorname{ad}}
\def\id{\mathrm{id}}

 \newtheorem{Theorem}{Theorem}[section]
\newtheorem{Corollary}[Theorem]{Corollary}
\newtheorem{Lemma}[Theorem]{Lemma}
\newtheorem{Proposition}[Theorem]{Proposition}
 { \theoremstyle{definition}
\newtheorem{Definition}[Theorem]{Definition}
\newtheorem{Example}[Theorem]{Example}
\newtheorem{Remark}[Theorem]{Remark} }

\numberwithin{equation}{section}

\begin{document}
\allowdisplaybreaks

\newcommand{\arXivNumber}{1605.01376}

\renewcommand{\PaperNumber}{026}

\FirstPageHeading

\ShortArticleName{Hopf Algebroid Twists for Deformation Quantization of Linear Poisson Structures}
\ArticleName{Hopf Algebroid Twists for Deformation Quantization\\ of Linear Poisson Structures}

\Author{Stjepan MELJANAC~$^\dag$ and Zoran \v{S}KODA~$^{\ddag\S}$}
\AuthorNameForHeading{S.~Meljanac and Z.~\v{S}koda}

\Address{$^\dag$~Theoretical Physics Division, Institute Rudjer Bo\v{s}kovi\'{c},\\
\hphantom{$^\dag$}~Bijeni\v{c}ka cesta~54, P.O.~Box 180, HR-10002 Zagreb, Croatia}
\EmailD{\href{mailto:meljanac@irb.hr}{meljanac@irb.hr}}
\Address{$^\ddag$~Faculty of Science, University of Hradec Kr\'{a}lov\'{e},\\
\hphantom{$^\ddag$}~Rokitansk\'{e}ho 62, Hradec Kr\'{a}lov\'{e}, Czech Republic}
\EmailD{\href{mailto:zoran.skoda@uhk.cz}{zoran.skoda@uhk.cz}}
\Address{$^\S$~University of Zadar, Department of Teachers' Education,\\
\hphantom{$^\S$}~Franje Tudjmana 24, 23000 Zadar, Croatia}

\ArticleDates{Received May 24, 2017, in final form March 13, 2018; Published online March 25, 2018}

\Abstract{In our earlier article~[\textit{Lett. Math. Phys.} \textbf{107} (2017), 475--503], we explicitly described a topological Hopf algebroid playing the role of the noncommutative phase space of Lie algebra type. Ping Xu has shown that every deformation quantization leads to a Drinfeld twist of the associative bialgebroid of $h$-adic series of differential operators on a fixed Poisson manifold. In the case of linear Poisson structures, the twisted bialgebroid essentially coincides with our construction. Using our explicit description of the Hopf algebroid, we compute the corresponding Drinfeld twist explicitly as a product
of two exponential expressions.}

\Classification{53D55; 16S30; 16T05}

\section{Introduction}\label{section1}

Given a possibly noncommutative associative $\mathbf{k}$-algebra $A$, a (left) $A$-bialgebroid $H$ is an associative algebra with an additional $A$-bimodule structure ${}_A H_A$, coproduct $\Delta \colon H\to H\otimes_A H$ and counit $\epsilon \colon H\to A$ generalizing appropriately the structures of a $\genfd$-vector space, coproduct and counit in the definition of a $\genfd$-bialgebra. We refer to $A$ as the base algebra and $H$ as the total algebra of the bialgebroid. For details see~\cite{bohmHbk,BrzMilitaru} and Section~\ref{sec:bialgebroids}. Lu~\cite{lu} introduced a class of $A$-bialgebroids meant to be noncommutative analogues of transformation groupoids, hence a novel symmetry useful in geometry and mathematical physics. These bialgebroids are smash products of a Hopf algebra and a braided commutative algebra in its category of Yetter--Drinfeld modules and nowadays they are often called scalar extension bialgebroids. A particular case is the smash product of a Hopf algebra and its dual (with its canonical Yetter--Drinfeld module structure); this smash product is known under the name of Heisenberg double and may require some completions in the infinite dimensional case. Our present focus is on the Lie algebra noncommutative phase spaces whose Hopf algebroid structure is presented in detail in~\cite{halg}, crudely shown to be a Heisenberg double of the universal enveloping algebra~$U(\gg)$ in~\cite{heisd}. An improved variant of this Heisenberg double and an entire class of generalizations are categorically treated in~\cite{stojicPhD}, namely the Heisenberg doubles (and more generally, scalar extension Hopf algebroids) in the internalized context of the symmetric monoidal category of (countably cofinite strict) filtered-cofiltered vector spaces, see~\cite{stojicPhD}. They are shown to be internal Hopf algebroids in the sense of B\"ohm~\cite{BohmInternal}.
It is compelling to interpret the Heisenberg double $H_\gg$ of the universal enveloping algebra $U(\gg)$ of a Lie algebra $\gg$ as a noncommutative phase space of Lie algebra type where the universal enveloping algebra $U(\gg)$ is interpreted as its coordinate sector (configuration space) and its (topological) Hopf dual as the momentum sector. In the case of $\kappa$-deformed Minkowski space, one can extend the Hopf algebroid adding full $\kappa$-deformed Poincar\'e symmetry into the Heisenberg double~\cite{JLZSMWPLB,JLZSMW}. The physical discussion of the coproduct for the momentum sector has been studied in many references including~\cite{Amelino,HallidaySzabo}.

Based on a work of Xu~\cite{xu}, this example can also be understood as a deformation quantization of the phase space with the corresponding linear Poisson structure. Xu defines an analogue of Drinfeld twist $\mathcal{F}\in H\otimes_A H$ twisting the coproduct $\Delta$ to a new coproduct $\Delta_{\mathcal{F}}$, by a recipe $\Delta_{\mathcal{F}}(h)=\mathcal{F}^{-1}\Delta(h)\mathcal{F}$, for all $h\in H$, where however the Hopf algebra $H$ from Drinfeld's theory is replaced by an $A$-bialgebroid $H$. On the other hand, if we consider the formal power series in one variable $t$ with coefficients in the Poisson algebra of functions, then the deformation quantization gives a recipe for an associative star product $f\star g$ which is a formal power series in $t$ whose zero term is the usual product of functions and the first correction gives the Poisson bracket. One usually restricts to the case when there exist a series $\mathcal{F}$ in $t$ whose coefficients are bidifferential operators such that the star product is of the form $f\star g = (\mu\circ \mathcal{F})(f\otimes g)$. Xu shows that such formal bidifferential operator $\mathcal{F}$ is in fact a twist for the topological bialgebroid whose total algebra consists of formal power series with coefficients in the original Poisson algebra of functions.

Excluding the bialgebroid twists induced by bialgebra twists~\cite{borpacholtwistedoid}, very few examples where both the deformation twist and the twisted bialgebroid structures are given by explicit formulas are known. The linear Poisson structures are clearly among the most important classes to study and here we find two explicit expressions~$\mathcal{F}_l$ and~$\mathcal{F}_r$ giving the same twist for the left bialgebroid with completed Heisenberg--Weyl algebra $\hat{A}_n$ to $H_\gg$ as the total algebra. Though $\mathcal{F}_l=\mathcal{F}_r$ as elements in $H_\gg\otimes_{U(\gg)} H_\gg$ the corresponding formulas come by projecting from two different elements $\tilde{\mathcal{F}}_l$, $\tilde{\mathcal{F}}_r$ in the usual tensor square $H_\gg\otimes H_\gg$. In~\cite{scopr} we exhibited another formula $\mathcal{F}_c$ for this twist which is however far less explicit and involves a series where each term involves both the original and twisted coproduct, see~\eqref{eq:Fc}.

Our main motivation is to find explicit examples of Xu's construction, and is reinforced by the recent interest in field theories on Lie algebra type noncommutative spaces, including $\kappa$-deformations.

A Hopf algebroid is defined in Section~\ref{sec:antipode} as a left bialgebroid with an antihomomorphism $S\colon H\to H$ called the antipode and satisfying some axioms. Unlike for the bialgebra twists, it is not known in general if the bialgebroid twists of Hopf algebroids are sufficient to formulate a~general recipe for twisting the antipode as well. In our case, both the bialgebroid before and the bialgebroid after the twist have an antipode. We give a conjectural formula $S(h) = V_{\mathcal{F}}^{-1}S_0(h)V_{\mathcal{F}}$ for the new antipode $S$ in terms of old antipode $S_0$ along with a partial argument for it where $V_{\mathcal{F}} = \mu(S_0\otimes\id)(\mathcal{F})$. We can show that $S(h) = V^{-1}S_0(h)V$ for some $V$, but we do not have the complete proof that $V_{\mathcal{F}} = V$. Both can be expressed as exponentials in formal power series of only the momentum variables $\partial^\mu$, starting with free term $1$ and agreeing in the first few terms. Note that Xu~\cite{xu} says Hopf algebroid for the notion
which does not involve an antipode and which is in~\cite{BrzMilitaru} shown to be equivalent to the left bialgebroid of other authors.

The representation of $H_\gg$ via a concrete twist, besides its conceptual appeal, is useful to twist systematically many other constructions (for example, basic constructions in differential geometry) from the undeformed Heisenberg--Weyl algebra case to the case of phase space of Lie algebra type. This is an additional tool for physical applications, e.g., development of field theories on the noncommutative spaces of Lie algebra type. Such applications are under investigation.

To orient the reader within the subject, we note that some other physically important examples of Hopf algebroids are built from the data of weak Hopf algebras~\cite{bohmHbk} (as those coming from the symmetries in low dimensional QFTs~\cite{mackschomerus}) and the study of the
dynamical quantum Yang--Baxter equation~\cite{doninmudrov,xu}.

{\bf Conventions.} In this article, all algebras are over a field $\genfd$ of characteristic $0$, and the unadorned tensor product $\otimes$ is over the ground field (in the deformation quantization and quantum gravity examples, the field of real or complex numbers). We freely use Sweedler
notation $\Delta(h) = \sum h_{(1)}\otimes h_{(2)}$ for the coproducts with or without explicit summation sign. If~$A$ is an algebra, $A^\op$ denotes the algebra with opposite multiplication. We use the Einstein summation convention.

\section{Bialgebroids}\label{sec:bialgebroids}

In this section, we define bialgebroids and some auxiliary constructions.
\begin{Definition}[\cite{bohmHbk,BrzMilitaru,lu}]
Given an associative algebra $(A,\mu_A)$, which is in this context called the {\it base algebra}, a {\it left $A$-bialgebroid} $(H,\mu,\alpha,\beta,\Delta,\epsilon)$ consists of
\begin{itemize}\itemsep=0pt
\item an associative algebra $(H,\mu)$;
\item two algebra maps, the {\it source map} $\alpha\colon A\to H$ and the {\it target map} $\beta\colon A^\op\to H$ such that $[\alpha(a),\beta(a')] = 0$ for all $a,a'\in A$; assume in the following that on $H$ we fix the $A$-bimodule
structure given by $a.h.a' = \alpha(a)\beta(a')h$, $a,a'\in A$, $h\in H$;

\item $A$-bimodule map $\Delta \colon H\to H\otimes_A H$, called {\it coproduct}, satisfying the coassociativity
\begin{gather*}
(\Delta\otimes_A \id_H)\circ\Delta = (\id_H\otimes_A \Delta)\circ\Delta;
\end{gather*}
\item $A$-bimodule map $\epsilon \colon H\to A$, called {\it counit}, satisfying $\alpha(\epsilon(h_{(1)}))h_{(2)} = h = \beta(\epsilon(h_{(2)}))h_{(1)}$.
\end{itemize}
The following compatibilities are required for these data:

(i) Formula $\sum_\lambda h_\lambda\otimes f_\lambda \mapsto \epsilon(\sum_\lambda h_\lambda\alpha(f_\lambda))$ defines an action $\blacktriangleright\colon H\otimes A\to A$ which extends the left regular action $A\otimes A\to A$ along the inclusion $A\otimes A\stackrel{\alpha\otimes A}\longrightarrow H\otimes A$.

(ii) The $A$-subbimodule $H\times_A H\subset H\otimes_A H$ (called the Takeuchi product~\cite{bohmHbk}) defined by
\begin{gather*}
H\times_A H = \left\lbrace\sum_i b_i\otimes b'_i \in H\otimes_A H \,|\,
\sum_i b_i\otimes b'_i \alpha(a) = \sum_i b_i \beta(a)\otimes b'_i, \ \forall\, a\in A \right\rbrace
\end{gather*}
contains the image of $\Delta$ and the corestriction $\Delta|\colon H\to H\times_A H$ is an algebra map with respect to the factorwise multiplication.
\end{Definition}

To see the meaning of~(ii), notice that, unlike for the Takeuchi product $H\times_A H$, the factorwise multiplication on $H\otimes H$ does not factor to a well defined map on the vector space $H\otimes_A H$ in general. Accordingly, it does not make sense to say that $\Delta\colon H\to H\otimes_A H$ is an algebra map and the multiplicativity of $\Delta$ should be interpreted differently, say within $H\times_A H$. Indeed, the kernel
\begin{gather}\label{eq:IA}
I_A = \operatorname{Ker} (H\otimes H\to H\otimes_A H )
\end{gather}
of the canonical projection of the $A$-bimodules is the right ideal in the algebra $H\otimes H$ generated by $\beta(a)\otimes 1 - 1\otimes\alpha(a)$ for all $a\in A$, but this kernel is not a two-sided ideal in general.

The compatibility (ii) in the definition of a bialgebroid holds iff the map $H\otimes(H\otimes_A H)\to H$, $(g,\sum_i h_i\otimes k_i)\mapsto\Delta(g)(\sum_i h_i\otimes k_i)$ (multiplied in each tensor factor) is well defined (does not depend on the choice of the sum within
the equivalence class in $H\otimes_A H = H\otimes H/I_A$). In other words, $\Delta(g)\cdot I_A\subset I_A$ for all $g\in H$. This last characterization will be useful in the proof of Proposition~\ref{prop:FDFm}.

\begin{Lemma}\label{lem:prodbtr}
The composition
\begin{gather}\label{eq:prodbtrprim}
(H\otimes H)\otimes (A\otimes A) \stackrel\cong\longrightarrow (H\otimes A) \otimes (H\otimes A)\stackrel{\blacktriangleright\otimes\blacktriangleright}\longrightarrow A\otimes A\stackrel{\mu_A}\longrightarrow A
\end{gather}
factors down along the projection $(H\otimes H)\otimes(A\otimes A) \to (H\otimes_A H)\otimes(A\otimes A)$ to a $\mathbf{k}$-linear map
\begin{gather*}
\mu_A\circ(\blacktriangleright\otimes\blacktriangleright)\colon \ (H\otimes_A H)\otimes (A\otimes A)\longrightarrow A.
\end{gather*}
By similar factorization, we get analogous maps from the multiple tensor products $(H\otimes_A\cdots\otimes_A H)\otimes(A\otimes\cdots\otimes A)\to A$. For elements we extend the operation syntax $h\blacktriangleright a := \blacktriangleright (h\otimes a)$ to products $\blacktriangleright\otimes\blacktriangleright$ (and multiple analogues), e.g., if $\mathcal{F}\in H\otimes_A H$ and $a,b\in A$ we write
\begin{gather*}\mu_A \mathcal{F}(\blacktriangleright\otimes\blacktriangleright)(a\otimes b):= (\mu_A\circ(\blacktriangleright\otimes\blacktriangleright))(\mathcal{F}
\otimes (a\otimes b)).\end{gather*}
\end{Lemma}
\begin{proof} We need to show that the composition~(\ref{eq:prodbtrprim}) gives zero on the subspace $I_A\otimes(A\otimes A)$. We show this in terms of the generators. For $a,b,c\in A$, $h,h'\in H$ we need
\begin{gather*}
\mu_A (((\beta(c) h)\blacktriangleright a)\otimes (h'\blacktriangleright b)) = \mu_A ((h\blacktriangleright a)\otimes((\alpha(c)h')\blacktriangleright b)).
\end{gather*}
By the definition of $\blacktriangleright$ (see (i) above), it is hence enough to show
\begin{gather*}
\mu_A (\epsilon(\beta(c) h \alpha(a))\otimes\epsilon(h'\alpha(b)))= \mu_A (\epsilon(h\alpha(a))\otimes\epsilon(\alpha(c)h'\alpha(b))).
\end{gather*}
Since $\epsilon$ is a map of $A$-bimodules, both sides yield $\epsilon(h\alpha(a)) \cdot c\cdot \epsilon(h' \alpha(b))\in A$.
\end{proof}

\section{Twists for bialgebroids}\label{s:gentwist}

In this section, we explain how Ping Xu~\cite{xu} generalized the Drinfeld twists to bialgebroids (see also~\cite{doninmudrov}), give some examples and focus on the case relevant for deformation quantization. This involves formal completions and completed tensor products and somewhat different formal completions in the later part of the article. For that reason, in this section, we also include an extensive review of completions adapted to our formalism (see also Remark~\ref{rem:wcompl}). Unlike in some other publications, we here use Xu's convention for twists (elsewhere we use~$\mathcal{F}$ for his~$\mathcal{F}^{-1}$).

\begin{Definition}[\cite{xu}]
$\mathcal{F}\in H\otimes_A H$ is a Drinfeld twist for a left $A$-bialgebroid
$(H,\mu,\alpha,\beta,\Delta,\epsilon)$
if the 2-cocycle condition
\begin{gather}\label{eq:2coc}
[(\Delta\otimes_A \mathrm{id})(\mathcal{F})](\mathcal{F}\otimes_A 1)
=
[(\mathrm{id}\otimes_A\Delta)(\mathcal{F})](1\otimes_A\mathcal{F})
\end{gather}
and the counitality $(\epsilon\otimes_A\mathrm{id})(\mathcal{F}) = 1_H
= (\mathrm{id}\otimes_A\epsilon)(\mathcal{F})$ hold.
\end{Definition}

We use the Sweedler-like notation for the twist $\mathcal{F} = \mathcal{F}^{(1)}\otimes\mathcal{F}^{(2)}$ (upper labels, while we use lower for the coproduct) and the notation from Lemma~\ref{lem:prodbtr}.
\begin{Theorem}[\cite{xu}]\label{th:cocycimpltw}
If $H$ is a left $A$-bialgebroid and $\mathcal{F}\in H\otimes_A H$ a Drinfeld twist, then the formula
\begin{gather}\label{eq:starfromF}
a\star b = \mu\mathcal{F}(\blacktriangleright\otimes\blacktriangleright)(f\otimes g) = \big(\mathcal{F}^{(1)}\blacktriangleright a\big) \big(\mathcal{F}^{(2)}\blacktriangleright b\big)
\end{gather}
defines an associative $\mathbf{k}$-algebra $A_\star = (A,\star)$ structure on $A$ with the same unit; the formulas $\alpha_{\mathcal{F}}(a) = \alpha\big(\mathcal{F}^{(1)}\blacktriangleright a\big)\mathcal{F}^{(2)}$ and $\beta_{\mathcal{F}}(a) = \beta\big(\mathcal{F}^{(2)}\blacktriangleright a\big)\mathcal{F}^{(1)}$ define respectively an algebra homomorphism $\alpha_{\mathcal{F}}\colon A_\star\to H$ and antihomomorphism $\beta_{\mathcal{F}}\colon A_\star\to H$ with $[\alpha_{\mathcal{F}}(a),\beta_{\mathcal{F}}(a)] = 0$ for all $a\in A_\star$. Given the $A_\star$-bimodule structure on $H$ by $a.h.b := \alpha_{\mathcal{F}}(a)\beta_{\mathcal{F}}(b)h$ and denoting by $I_{A_\star}$ the kernel of the canonical projection $H\otimes H\to H\otimes_{A_\star} H$, there is an inclusion $\mathcal{F} \cdot I_{A_\star}\subset I_A \subset H\otimes H$ $($cf.~\eqref{eq:IA}$)$. In other words, $\mathcal{F}$ induces a linear map $\mathcal{F}^\sharp\colon H\otimes_{A_\star} H\to H\otimes_{A}H$ given by the left multiplication of classes in $H\otimes H$ by $\mathcal{F}$, that is by $\mathcal{F}^\sharp(a\otimes_{A_\star} b) = f^1 a\otimes_{A} f_1 b$. Suppose the map~$\mathcal{F}^\sharp$ is invertible. Then the formula
\begin{gather*}
\Delta_{\mathcal{F}}(h) = \mathcal{F}^{\sharp -1}\Delta(h)\mathcal{F}
\end{gather*}
defines a map $\Delta_{\mathcal{F}}\colon H\to H\otimes_{A_\star} H$ which is coassociative and counital with the same counit. Moreover, $H_{\mathcal{F}} = (H,\mu,\alpha_{\mathcal{F}},\beta_{\mathcal{F}},\Delta_{\mathcal{F}},\epsilon)$ is a~left $A_\star$-bialgebroid.
\end{Theorem}

\begin{Remark}[inverse cocycle] \label{rem:invcoc}
For $\mathcal{F}^{\sharp -1}\colon H\otimes_{A} H\to H\otimes_{A_\star}H$ to exist, it is sufficient that there is an element $\mathcal{F}^{-1}\in H\otimes_{A_\star} H$ such that $\mathcal{F}^{-1}\cdot I_A \subset I_{A_\star}$ and the identities $\mathcal{F}^{-1}\mathcal{F} = 1\otimes 1 + I_{A_\star}$ and $\mathcal{F}\mathcal{F}^{-1} = 1\otimes 1 + I_{A_\star}$ hold in $H\otimes H$. In that case, $\mathcal{F}^{\sharp -1} = \mathcal{F}^{-1 \sharp}$. This will be the situation throughout the article, hence
we do not need to distinguish $\mathcal{F}^{\sharp -1}$ from the left multiplication with $\mathcal{F}^{-1}$.
In terms of $\mathcal{F}^{-1}$ and $A_\star$,
the cocycle condition~(\ref{eq:2coc}) is equivalent to the condition
\begin{gather*}
\big(\mathcal{F}^{-1}\otimes_{A_\star} 1\big)(\Delta\otimes_{A_\star}\mathrm{id})\big(\mathcal{F}^{-1}\big)
= \big(1\otimes_{A_\star}\mathcal{F}^{-1}\big)(\mathrm{id}\otimes_{A_\star}\Delta)\big(\mathcal{F}^{-1}\big).
\end{gather*}
\end{Remark}

\begin{Remark}\label{rem:cocyccoass} The proof of the associativity of $\star$ in the theorem follows from comparing $(a\star b)\star c$ and $a\star(b\star c)$, which are from the definition~(\ref{eq:starfromF}) easily calculated to be
\begin{gather*}
(a\star b) \star c = \mu(\mu\otimes\id)[(\Delta\otimes\id)\mathcal{F}](\mathcal{F}\otimes\id)(\blacktriangleright\otimes\blacktriangleright\otimes\blacktriangleright)(a\otimes b\otimes c),
\\
a\star (b \star c) = \mu(\id\otimes\mu)[(\id\otimes\Delta)\mathcal{F}](\id\otimes \mathcal{F})(\blacktriangleright\otimes\blacktriangleright\otimes\blacktriangleright)(a\otimes b\otimes c).
\end{gather*}
Thus the cocycle condition~(\ref{eq:2coc}) implies the coassociativity. The converse does not hold for arbitrary bialgebroids. Indeed, let $\mu_2 = \mu(\mu\otimes\id) = \mu(\id\otimes\mu)$ be the second iterate of the multiplication and $\mathcal{F}_{3a}$ and $\mathcal{F}_{3b}$ the left and the right-hand sides of the cocycle condition~(\ref{eq:2coc}); we can see only that the coassociativity implies $\mu_3(\mathcal{F}_{3a}-\mathcal{F}_{3b})(\blacktriangleright\otimes \blacktriangleright\otimes\blacktriangleright)(a\otimes b\otimes c) = 0$
for all $a,b,c\in H$. This implies $\mathcal{F}_{3a}-\mathcal{F}_{3b} = 0$ in $H\otimes_A H\otimes_A H$ if $\blacktriangleright\otimes\blacktriangleright \otimes\blacktriangleright$ is nondegenerate in the first argument, which is not true for
an arbitrary bialgebroid. The nondegeneracy holds for the undeformed Heisenberg algebra and by Theorem~\ref{th:th4halg} more generally for the $U(\gg)$-bialgebroid $H_\gg$ studied in this paper~-- this is the content of Theorem~\ref{th:th4halg} with the consequence for the cocycle condition, Corollary~\ref{cor:FDFimpliesDt}.
\end{Remark}

\begin{Example}[basic example of a noncommutative bialgebroid over a {\it commutative} base]\label{ex:basic} Let $A = C^\infty(M)$, where $M$ is a smooth real manifold and let $H=\mathcal{D}$ be the algebra of global {\it differential operators} with smooth coefficients. The formula $\Delta(D)(f,g) = D(f\cdot g)$ defines a~bidifferential operator $\Delta(D)\in\mathcal{D}\otimes_{C^\infty(M)}\mathcal{D}$; this defines a~cocommutative coproduct~$\Delta$. The base $C^\infty(M)$ is commutative and $\alpha = \beta$ is the canonical embedding of the algebra of functions into the algebra of differential operators.

This example generalizes. The algebra of differential operators $\mathcal{D}$ is canonically isomorphism to the universal enveloping algebra of the Lie algebroid given by the tangent bundle $V = T M$ where the anchor map $a$ is the identity. Universal enveloping algebras of Lie algebroids are bialgebroids in a canonical way. To explain this in few lines let us say that a Lie algebroid $L = (V,a,[\,,\,])$ is a vector bundle $V$ over a manifold $M$ together with a map $a\colon V\to T M$ of vector bundles (called the anchor) and an antisymmetric $\mathbb{R}$-bilinear bracket $[\,,\,]$ on the $C^\infty(M)$-module of global sections $\Gamma V$. It is required that the map, $a_*\colon \Gamma V\to\Gamma T M$ is a map of $\mathbb{R}$-Lie algebras (where the bracket on $\Gamma T M$ is the Lie bracket of vector fields) and is satisfying the Leibniz rule $[X,f Y] = f [X,Y] + a(X)(f)Y$.
For any Lie algebroid $V$ one defines the universal enveloping algebra $U(L)$ as the tensor algebra of the space of sections $\Gamma V$ modulo the ideal generated by the relations $[X,Y] - X\otimes Y + Y\otimes X$, and $X f Y - a(X)(f)Y - f X Y$, where $f\in C^\infty(M)$, $X,Y\in\Gamma V$. It becomes a bialgebroid over the commutative algebra $C^\infty(M)$ via the formula $\Delta(f) = f\otimes 1$ for $f\in C^\infty(M)\subset U(V)$ and $\Delta(X) = X\otimes 1 + 1\otimes X$ for $X\in\Gamma V\subset U(V)$.

Thesis~\cite{kowalzig} settled the question when bialgebroid structures on universal enveloping algebras of Lie algebroids, and more generally of Lie--Rinehart algebras, have an antipode (see Section~\ref{sec:antipode}), namely iff a flat connection in the appropriate context exists. In~\cite{xu} an erroneous argument (confusing the role of the antipode with the duality for functional spaces) was given that the bialgebroids of differential operators can not have an antipode. The special case of the Heisenberg--Weyl algebra in this article of course has an antipode.
\end{Example}

\begin{Remark}[notation on completions] Given a vector space $V$ an $\mathbb{N}$-cofiltration of $V$ is given by specifying for all $n\in\mathbb{N}$ the quotient spaces $V_n$ called cofiltering components with projections $\pi_n\colon V\to V_n$, connecting quotient maps $\pi_{n m}\colon V_{m}\to V_n$ for all $n\leq m$, satisfying $\pi_{n m}\circ\pi_{m k} = \pi_{n k}$ for all $n\leq m\leq k$, and commuting with the projections, $\pi_{n m}\circ\pi_m=\pi_n$. One considers the completion $\hat{V}$ as the inverse limit $\underset\longleftarrow\lim{}_n V_n$ which may be realized by threads, that is sequences $(v_n)_{n\in\mathbb{N}}$ where $v_n\in V_n$ and $\pi_{n m}(v_m) = v_n$ for all $n\leq m$. The canonical map $V\to\hat{V} = \underset\longleftarrow\lim{}_n V_n$, $v\mapsto (\pi_n(v))_n$ is injective iff for every $v\in V$ there is $n\in\mathbb{N}$ such that $\pi_n(v)\neq 0$. In that case, we usually identify $V$ with its image within $\hat{V}$. If the canonical map $V\to\hat{V}$ is an isomorphism we say that $V$ is complete. A~cofiltration is often given by a~decreasing filtration, that is a~decreasing sequence of subspaces $V'_n$ with empty common intersection; the quotients are given by $V_n = V/V'_n$ and the inclusions $V'_n\hookrightarrow V'_{n+1}$ induce the quotient maps $V/V'_{n+1}\to V/V'_n$. If all~$V'_n$ are of finite codimension in $V$, then the completion may be understood in the topological sense (consisting of all equivalence classes of Cauchy sequences), where~$V$ is understood as a topological vector space with linear topology for which the subspaces~$V'_n$ form a basis of neighborhoods of $0$. Given two cofiltered spaces $V$ and $W$, there is a completed tensor product also adorned with the hat symbol, $V\hat\otimes W = \underset\longleftarrow\lim{}_{n,m} V_n\otimes W_m$, where the inverse limit is over the directed set $\mathbb{N}\times\mathbb{N}$; the spaces $V_n\otimes V_m$ together with the connecting morphisms form an $\mathbb{N}\times\mathbb{N}$-cofiltration~\cite{stojicPhD}, which can be replaced, in the generality of this paper, by an equivalent $\mathbb{N}$-cofiltration~\cite[Appendix~2]{halg}. Given a cofiltered vector space $V$, an expression $\sum_{\lambda\in\Lambda} v_\lambda$ where $v_\lambda\in V$ is a profinite (or formal) sum~\cite{stojicPhD} if for each $n$ there are only finitely many $\lambda\in\Lambda$ such that $\pi_n(v_\lambda)\neq 0$. Then the finite sums $s_n = \sum_{\lambda, \pi_n(\lambda)\neq 0}\pi_n(v_\lambda)\in V_n$ define a thread $(s_n)_{n\in\mathbb{N}}$ which is therefore an element of the completion $\hat{V}$ called the value $v$ of the profinite sum $\sum_{\lambda\in\Lambda} v_\lambda$ and we write $v = \sum_{\lambda\in\Lambda} v_\lambda$. A linear map $f\colon V\to W$ between cofiltered vector spaces distributes over profinite sums (equivalently: is a cofiltered map) if for each profinite sum $\sum_{\lambda\in\Lambda} v_\lambda$ in $V$ the expression $\sum_{\lambda\in\Lambda} f(v_\lambda)$ is profinite in $W$ and the equation $f\big(\sum_{\lambda\in\Lambda} v_\lambda\big) = \sum_{\lambda\in\Lambda} f(v_\lambda)$ holds for the values; cofiltered spaces (even if noncomplete) with cofiltered maps form a monoidal category with the usual tensor product $\otimes$. Similarly, the subcategory of complete cofiltered spaces is a monoidal category with respect to the completed tensor product $\hat\otimes$. Each cofiltered map $f\colon V\to W$ induces the morphisms among the limits of cofiltrations, which can be interpreted as the completed cofiltered map $\hat{f}\colon \hat{V}\to\hat{W}$, the completion of $f$, which commutes with the canonical maps $i_V\colon V\to\hat{V}$, $i_W\colon W\to\hat{W}$ into the completions, $\hat{f}\circ i_V = i_W\circ f$. It can be seen that a subspace $Z$ of a complete cofiltered space $V\cong\underset\longleftarrow\lim{}_\lambda V_\lambda$ is complete iff for all profinite sums whose summands are in $Z$ their values are also in $Z$. The cofiltered maps among the complete vector spaces for which the topological interpretation holds, coincide with the continuous linear maps. Algebra in the monoidal category of cofiltered vector spaces with the usual tensor product is said to be a noncomplete cofiltered algebra. The modules and ideals of complete algebras should be completed as well, the fact which we often pass over silently. In particular, if~$A$ is a noncomplete cofiltered algebra, $M$~a~right $A$-module and~$N$ a left $A$-module, such that the actions are cofiltered maps, the kernel $I_A = \Ker(M\otimes N\to M\otimes_A N)$ is a~subspace $I_A\subset M\otimes N\hookrightarrow M\hat\otimes N$ and its completion is identified with a subspace $\hat{I}_A\subset M\hat\otimes N$. If the actions on $M$ and $N$ are cofiltered maps, then the vector space $M\hat\otimes_A N = M\hat\otimes N/\hat{I}_A$ is cofiltered, and called the completed tensor product of~$M$ and~$N$ over~$A$, see~\cite{stojicPhD}. The multiplication of~$A$ is cofiltered, hence it extends to a cofiltered map $\hat{A}\otimes\hat{A}\to\hat{A}$ and even to a unique cofiltered map $A\hat\otimes A\cong \hat{A}\hat\otimes\hat{A}\to\hat{A}$; thus $\hat{A}$ becomes an algebra in the category of complete cofiltered vector spaces with the completed tensor product. In this situation, if $M$ is a right $A$-module and the action is a cofiltered map $A\otimes M\to M$, then $\hat{M}$ is a right complete cofiltered $\hat{A}$-module, in other words the action is a cofiltered map $\hat{A}\hat\otimes\hat{M}\to\hat{M}$ (note the complete tensor product), hence by restriction a fortiori an $A$-module. It is straightforward to observe that $\hat{M}\hat\otimes_{\hat{A}}\hat{N}\cong\hat{M}\hat\otimes_A\hat{N}\cong M\hat\otimes_A N$ (indeed, $\hat{I}_{\hat{A}} = \hat{I}_A$ as a subset of $M\hat\otimes N\cong\hat{M}\hat\otimes\hat{N}$). Therefore, if $U$ is the (faithful) forgetful functor from the category of $\hat{A}$-bimodules with the completed tensor product to the category of the usual $A$-bimodules, then there is a natural isomorphism of functors $\hat\otimes_A \circ (U\times U)\cong U\circ \hat\otimes_{\hat{A}}$. The following is a simple example of our interest. If $t$ is a formal variable and $B$ an algebra then the completion of the polynomial ring $B[t]$ with respect to the cofiltrations by the finite dimensional subspaces $B[t]/(t^n B[t])$ of truncations modulo $t^n$ is the formal series ring $B[[t]]$. Then, the restriction of the (bi)functor $\otimes_{B[t]}$ to the subcategory of complete $B[t]$-modules (with the usual $\otimes$) such that the actions are cofiltered, agrees in the above sense with the (bi)functor $\otimes_{B[[t]]}$. In particular, we can move the formal series from $B[[t]]$ across the tensor product as a limiting case of moving its finite truncations in $B[t]/(t^n B[t])\cong B[[t]]/(t^n B[[t]])$ at the level of cofiltered components. Note a minor difference between the two (bi)functors, namely for $\otimes_{B[[t]]}$ we work with the $B[[t]]$-bimodule structures in the monoidal category of complete cofiltered spaces with the $\hat\otimes$-tensor product, rather than with the usual $B[t]$-bimodule structures.

Unlike in~\cite{halg}, we do not pay full attention to completions.
\end{Remark}

\begin{Example}[an algebraic analogue of the basic Example~\ref{ex:basic} of a noncommutative bialgebroid over commutative base] \label{ex:varbasic}
Here $A=\mathcal{O}(M)$ is the algebra of regular functions on a smooth affine variety over a field $\genfd$ of characteristic $0$ and $H=\mathcal{D}$ is the algebra of regular differential operators. Again, $\Delta(D)(f,g) = D(f\cdot g)$, the base $A$ is commutative and $\alpha = \beta$ is the canonical inclusion of elements of $A$ as the operators of left multiplication. Here and in Example~\ref{ex:basic} the counit is taking the constant term and $\blacktriangleright$ is the usual action of differential operators on functions (below denoted $\triangleright$ or as evaluation $D(f) = D\triangleright f$). If~$M$ is the affine $n$-space, $H$ is the Weyl algebra $A_n$ and the corresponding coproduct $\Delta_0$ is primitive on the generators $\partial^1,\dots,\partial^n$ and $\Delta_0(x_\mu) = x_\mu$ on the remaining generators $x_1,\dots,x_n$ of the Weyl algebra. This $\mathcal{O}(M)$-bialgebroid is in fact a~Hopf algebroid with the antipode $S(x_\mu) = x_\mu$ and $S(\partial^\nu) = -\partial^\nu$. We may complete $A_n$ by the degree of a differential operator to the algebra $\hat{A}_n$ and the Hopf algebroid structure continuously extends~(\cite{halg}) to a topological coproduct on $\hat{A}_n$ with values in the completed tensor product $\hat{A}_n\hat\otimes_A\hat{A}_n$ relative over $A = \genfd[x_1,\dots,x_n]$,
\begin{gather}\label{eq:Delta0}
\Delta_0 \colon \ \hat{A}_n\to\hat{A}_n\hat\otimes_{\genfd[x_1,\dots,x_n]}\hat{A}_n,
\qquad \Delta_0(x_\mu) = x_\mu\otimes 1, \qquad \Delta_0(\partial^\nu) = 1\otimes\partial^\nu+\partial^\nu\otimes 1.
\end{gather}
By abuse of notation, we denote by $\Delta_0$ also the restriction of $\Delta_0$ to the polynomial algebra of $\partial^1,\dots,\partial^n$ (or to its completion $\genfd[[\partial^1,\dots,\partial^n]]$, the algebra of formal power series), which therefore becomes a Hopf $\genfd$-algebra (respectively, a topological Hopf $\genfd$-algebra).
\end{Example}

\begin{Example}[related to {\it deformation quantization}]\label{ex:defquant} Ping Xu~\cite{xu} extends the base algebra $A = C^\infty(M)$
from Example~\ref{ex:basic} or $A = \genfd[x_1,\dots,x_n]$ from Example~\ref{ex:varbasic} to the algebra of formal power series $A[[h]]$ in a formal variable $h$ with coefficients in $A$. Then $\mathcal{D}[[h]]$ carries a left $A[[h]]$-bialgebroid structure obtained from $A$-bialgebroid $\mathcal{D}$ by extending the scalars from $A$ to $A[[h]]$ and one works with the $h$-adically completed tensor product $\hat\otimes_{\genfd[[h]]}$.
\end{Example}

\begin{Theorem}[\cite{xu}]\label{thm:defqinterp}
Suppose $M$ is a Poisson manifold and a formal bidifferential operator $\mathcal{F} \in (\mathcal{D}\otimes_{C^\infty(M)}\mathcal{D})[[h]]$ is given. If $\mathcal{F}$ defines a deformation quantization of $M$ via the star product $\mu\mathcal{F}(f\otimes g)$ then $\mathcal{F}$ is a Drinfeld twist for the left $C^\infty(M)[[h]]$-bialgebroid $\mathcal{D}[[h]]$ of formal power series in regular differential operators. Consequently, by Theorem~{\rm \ref{th:cocycimpltw}}, each deformation quantization of $M$ defines also a twist deformation of the bialgebroid $\mathcal{D}[[h]]$, where the twisted base algebra is $C^\infty(M)[[h]]$, equipped with the star product $($both in the sense of constructed from a bialgebroid twist and in the sense of deformation quantization$)$.
\end{Theorem}
We are interested in using the bialgebroid techniques to find explicit formulas for $\mathcal{F}$ and also to describe the Xu's bialgebroid in detail in special cases. There are rather few explicit formulas for bialgebroid twists in the literature which are not induced from bialgebra twists
by a procedure studied in~\cite{borpacholtwistedoid}.

\section{Phase spaces of Lie type as bialgebroids}\label{s:repeathalg}

In this section, we present the bialgebroid $H_\gg$ from our earlier article~\cite{halg}. This bialgebroid is in fact a Hopf algebroid, but the discussion of the antipode is deferred to Section~\ref{sec:antipode}.

Throughout, $\gg$ is a fixed Lie algebra over $\genfd$ with basis $\hx_1,\dots,\hx_n$, $U(\gg)$ is the universal enveloping algebra and $S(\gg)$ the symmetric algebra of~$\gg$. The generators of $U(\gg)$ are also denoted $\hx_1,\dots,\hx_n$ (and viewed as noncommutative coordinates), but the corresponding generators of~$S(\gg)$ are $x_1,\dots,x_n$ (and viewed as coordinates on the undeformed commutative space). The structure constants $C_{\mu\nu}^\lambda$ are given by
\begin{gather*}
[\hx_\mu,\hx_\nu] = C_{\mu\nu}^\lambda \hx_\lambda.
\end{gather*}
Let $\partial^1,\dots,\partial^n$ be the dual basis of $\gg^*$, which are also (commuting) generators of $S(\gg^*)$. Let $\hat{S}(\gg^*)$ be the formal completion of $S(\gg^*)$ (which should be interpreted as the algebra of formal functions on the space of deformed momenta). A basis of neighborhoods of $0$ is made out of powers of the ideal $J = \gg\cdot\hat{S}(\gg^*)$ generated by $\gg$ ($J$-adic topology of the formal power series ring). The ring $M_n(\hat{S}(\gg^*))$ of $n\times n$-matrices with entries in $\hat{S}(\gg^*)$ is rank $n^2$ free $\hat{S}(\gg^*)$-module and as such it inherits the topology of componentwise convergence. By a straightforward check we can see that this topology is equivalent to the $M_n(J)$-adic topology, where $M_n(J) = M_n(\gg\cdot\hat{S}(\gg^*))$ consists of all matrices with entries in $\gg\cdot\hat{S}(\gg^*)$. Therefore $M_n(\hat{S}(\gg^*)$ is a linearly compact topological ring as well. We introduce an auxiliary matrix $\CC\in M_n(\hat{S}(\gg^*))$ with entries
\begin{gather}\label{CCmatrix}
\CC^\alpha_\beta := C^\alpha_{\beta\gamma}\partial^\gamma \in \hat{S}(\gg^*),
\end{gather}
with the summation over repeated Greek indices and where, for the purposes of this article, we use the convention that $\alpha$ is the row and $\beta$ the column index. In this notation, we introduce matrices
\begin{gather}\label{eq:phi}
\phi := \frac{-\CC}{e^{-\CC} -1} =\sum_{N=0}^\infty \frac{(-1)^N B_N}{N!} \CC^N, \qquad \tilde{\phi} := \frac{\CC}{e^{\CC} -1} =\sum_{N=0}^\infty \frac{B_N}{N!} \CC^N,
\end{gather}
where $B_N$ are the Bernoulli numbers. By a simple comparison of the expressions~(\ref{eq:phi}) we obtain
\begin{gather}\label{eq:phiphiO}
\tilde\phi_\alpha^\beta = \phi^\alpha_\rho \OO^{-1\rho}_\beta,\qquad \text{where}\qquad \OO := e^\CC\in M_n\big(\hat{S}(\gg^*)\big).
\end{gather}
By $\hat{A}_n$ denote the completion by the degree of differential operators of the $n$-th Weyl algebra $A_n$ with generators $x_1,\dots,x_n$, $\partial^1,\dots,\partial^n$ hence the underlying vector space of $A_n$ is $S(\gg)\otimes S(\gg^*)$ and $\hat{A}_n$ is some completion of it. The symmetric algebra $S(\gg)$ is a~Hopf algebra which acts on~$\hat{S}(\gg^*)$ by a~unique action ${\boldsymbol \delta}\colon S(\gg)\to\End_\genfd(\hat{S}(\gg^*))$ which is a Hopf action, namely
\begin{gather*}{\boldsymbol \delta}(f)(P\cdot Q)={\boldsymbol \delta}(f_{(1)})(P){\boldsymbol \delta}(f_{(2)})(Q)\qquad \text{and}\qquad {\boldsymbol \delta}(f)(1) = \epsilon(f) 1, \end{gather*}
and which satisfies ${\boldsymbol \delta}(x_\nu)(\partial^\nu) = \delta^\nu_\mu$ on the generators. It is a useful point of view for the deformations below that the product in $A_n$ is the multiplication of the smash product algebra given by the formula
\begin{gather*}(f\otimes P) (g\otimes Q) = \sum f\cdot g_{(1)}\otimes {\boldsymbol \delta}(g_{(2)})(P)\cdot Q\qquad \text{for}\quad f,g\in S(\gg),\quad P,Q\in\hat{S}(\gg^*).\end{gather*}
The algebra $\hat{A}_n$ acts on the symmetric algebra $S(\gg)$ of polynomials in $x_1,\dots,x_n$, extending the Fock action of $A_n$ on $S(\gg)$ by differential operators $(P,f)\mapsto P(f)$.

Now define the elements $\hat{x}^\phi,\hat{y}^\phi\in\hat{A}_n$
\begin{gather}\label{eq:xyphi}
\hat{x}^\phi_\rho := \sum_\tau x_\tau \phi^\tau_\rho,\qquad \hat{y}^\phi_\rho := \sum_\tau x_\tau\tilde\phi^\tau_\rho.
\end{gather}
With our choice of the matrix ${\boldsymbol \phi}$, the right-hand side in the equations~\eqref{eq:xyphi} are dual, as explained and derived in~\cite{halg}, Section~\ref{sec:bialgebroids}, to the expressions for the left and right invariant vector fields respectively, corresponding to the Lie algebra generators, in the chart given by the exponential map. The expressions for $\phi^\tau_\rho$, $\hat{x}^\phi_\rho$ and $\hat{y}^\phi_\rho$ were derived many times historically, most notably by Berezin in the geometric context~\cite{berezin} and the corresponding expression for the star product studied in $C^\infty$-context has been introduced by Gutt~\cite{gutt}; the supersymmetric generalization of the formula in the algebraic context and in the dual language of coderivations (rather than derivations and vector fields here) is in~\cite{Petracci}; finally the article~\cite{ldWeyl} gives three proofs of the formula valid over any ring containing rationals: a direct calculation, a proof in formal geometry over arbitrary rings and an algebraic proof close in spirit to the approach in~\cite{Petracci}. The expression for~${\boldsymbol \phi}$ corresponds to the part of Baker--Hausdorff series linear in one of the variables. The same star product is often defined for exponentials via the full Baker--Hausdorff series~\cite{Amelino,HallidaySzabo,exp}.

From~(\ref{eq:phiphiO}) it follows immediately that
\begin{gather}\label{eq:yeqxO}
\hat{y}^\phi_\alpha = \hat{x}^\phi_\beta \mathcal{O}^{-1\beta}_\alpha,
\end{gather}
and one can also prove (see, e.g., \cite[Appendix~1]{halg})
\begin{gather}\label{eq:commxy}
[\hat{x}^\phi_\alpha,\hat{y}^\phi_\beta] = 0.
\end{gather}
In the geometric interpretation, equation~\eqref{eq:commxy} means that all left invariant vector fields on a Lie group commute with all right invariant vector fields. Then $\hx_\rho\mapsto\hx_\rho^\phi$ extends to a unique algebra map $(-)^\phi\colon U(\gg)\to\hat{A}_n$. This {\it realization} map $(-)^\phi$ is related to the symmetrization (PBW) isomorphism
\begin{gather*}
\xi\colon \ S(\gg)\cong U(\gg),\qquad w_1\cdots w_r\mapsto \frac{1}{r!}
\Sigma_{\sigma\in \Sigma(r)} \hat{w}_{\sigma 1}\cdots \hat{w}_{\sigma r}, \qquad w_i\in\gg\hookrightarrow U(\gg),
\end{gather*}
where $\Sigma(r)$ is the symmetric group on $r$ letters, in the sense that $\xi^{-1}(u)= u^\phi(1)$ ($\phi$-realization of $u$ acting on $1$) and more generally $u\cdot\xi(f) = u^\phi(f)$ for all $u\in U(\gg),f\in S(\gg)$, hence the star product $f\star g:=\xi^{-1}(\xi(f)\cdot_{U(\gg)}\xi(g))$ on $\gg$ may be written as $f\star g = \xi(f)^\phi(g)$. Thus, in physics literature, our choice of~$\phi$ is said to correspond to the symmetric ordering. Analogously, other coalgebra isomorphisms $S(\gg)\cong U(\gg)$ identical on $\gg$ correspond to different choices of~$\phi$ or to different orderings, see~\cite{scopr}. Our~$\phi$ corresponds to the symmetric ordering~\cite{ldWeyl}. The formula ${\boldsymbol \phi}(\hat{x}_\mu)(\partial^\nu) = \phi^\nu_\mu$ defines a linear map ${\boldsymbol \phi}(\hat{x}_\mu)\colon \gg^*\to\hat{S}(\gg^*)$ which extends by the chain rule to a (equally denoted) unique continuous derivation ${\boldsymbol \phi}(\hat{x}_\mu)\colon \hat{S}(\gg^*)\to\hat{S}(\gg^*)$. Denote by $\Derc(\hat{S}(\gg^*))$ the Lie algebra of all continuous derivations of $\hat{S}(\gg^*)$. A key property of $\phi$ is that the linear map ${\boldsymbol \phi}\colon \gg\to\Derc(\hat{S}(\gg^*))$ extending the rule $\hat{x}_\mu\mapsto{\boldsymbol \phi}(\hat{x}_\mu)\in\Derc(\hat{S}(\gg^*))$ is a Lie algebra homomorphism, hence it further extends to a unique right Hopf action
\begin{gather*}
{\boldsymbol \phi}\colon \ U(\gg)\to \End^{\mathrm{op}}\big(\hat{S}(\gg^*)\big),
\end{gather*}
which we denoted by the same symbol. The property that this action ${\boldsymbol \phi}$ is Hopf ensures that we can define the smash product algebra
\begin{gather*}
H_\gg = U(\gg)\sharp_{{\boldsymbol \phi}}\hat{S}(\gg^*)
\end{gather*}
in the usual way: it is the tensor product of vector spaces $U(\gg)\otimes\hat{S}(\gg^*)$ with the multiplication
\begin{gather*}(u\sharp f) (v\sharp g) = \sum u\cdot v_{(1)}\sharp {\boldsymbol \phi}(v_{(2)})(f)\cdot g\qquad \text{for}\quad u,v\in U(\gg),\quad f,g\in\hat{S}(\gg^*),\end{gather*}
where, for the emphasis, the expression $u\sharp f$ denotes $u\otimes f$ as an element in the smash product (whenever the meaning is clear the simpler product notation $u f$ may be used as well). One can check that the algebra map $H_\gg\to\hat{A}_n$, which on elements of $U(\gg)\sharp\genfd \hookrightarrow H_\gg$ agrees with the realization $\hat{x}_\mu\mapsto\hat{x}^\phi_\mu$ and sends $1\sharp\partial^\beta\in \genfd\otimes\hat{S}(\gg^*)\subset H_\gg$ to $\partial^\beta\in A_n$, is an algebra isomorphism; hence from now on this algebra map $H_\gg\to\hat{A}_n$ is viewed as an identification and sometimes called realization as well. We thus identify $\hat{x}_\mu\in U(\gg)$ and $\hat{x}^\phi_\mu\in \hat{A}_n$ etc.

The isomorphism of coalgebras $\xi\colon S(\gg)\to U(\gg)$ induces the dual isomorphism of algebras $\xi^T\colon U(\gg)^*\to S(\gg)^*\cong \hat{S}(\gg^*)$ hence the map $\mu^T$ dual to the multiplication $\mu\colon U(\gg)\otimes U(\gg)\to U(\gg)$ can be identified with a deformed comultiplication $\Delta_{\hat{S}(\gg^*)} = \big(\xi^T\otimes\xi^T\big)\circ\mu^T\colon \hat{S}(\gg^*) \to\hat{S}(\gg^*)\hat\otimes\hat{S}(\gg^*)$ (the completion on $\otimes$ comes from (co)filtrations, see~\cite{halg}). The source map $\alpha\colon U(\gg)\to H_\gg$ is given by $u\mapsto u\sharp 1$. Introduce
\begin{gather*}
\hy_\mu := \hx_\lambda\sharp\mathcal{O}^\lambda_\mu\in U(\gg)\sharp_\phi\hat{S}(\gg^*) = H_\gg.
\end{gather*}
In the realization $H_\gg\cong\hat{A}_n$, this expression becomes $\hy_\mu^\phi$. The rule $\hx_\mu\mapsto\hy_\mu$ extends to a unique algebra map $\beta\colon U(\gg)^\op\to H_\gg$, the target map. Equation~(\ref{eq:commxy}) implies that $[\alpha(u),\beta(v)] = 0$ for all $u,v\in U(\gg)$. We also often identify $u$ with its image $\alpha(u)$ and $P\in\hat{S}(\gg^*)$ with $1\sharp P\in H_\gg$. Thus, $H_\gg$ becomes a $U(\gg)$-bimodule via $u.h.v = \alpha(u)\beta(v)h$. It has been shown in~\cite{halg} that, with the appropriate completions, $H_\gg$ is a formally completed version of a Hopf $U(\gg)$-algebroid with the coproduct
\begin{gather*}
\Delta_H \colon \ H_\gg\to H_\gg\hat\otimes_{U(\gg)}H_\gg,\qquad
u\sharp P\stackrel{\Delta_H}\mapsto u \sharp P_{(1)}\otimes_{U(\gg)} P_{(2)},\\
\Delta_{\hat{S}(\gg^*)}(P) = P_{(1)}\otimes P_{(2)},\qquad u\in U(\gg), \qquad P\in\hat{S}(\gg^*).
\end{gather*}

Notice that in the undeformed case (that is, when $\gg$ is Abelian) the (topological) bialgebroid~$H_\gg$ is some completion of the Weyl algebra $A_n$ of regular differential operators on $\genfd^n$ where $A_n$ is viewed as a bialgebroid over the commutative base algebra consisting of polynomial functions on $\genfd^n$. Hence, the undeformed~$H_\gg$ is very close to the bialgebroid of differential operators over a smooth manifold.

{\bf Notation.} In the case of deformed bialgebroid $H = H_\gg$, the action $\blacktriangleright$ will be denoted
\begin{gather*}
\blacktriangleright_\gg\colon \ H_\gg\otimes U(\gg)\to U(\gg),\end{gather*}
and in the undeformed case $\triangleright$. The action $\triangleright \colon H_\gg\otimes S(\gg)\to S(\gg)$ is the ``Fock action''of $\hat{A}_n$ by differential operators on polynomials precomposed by the realization isomorphism $H_\gg\cong\hat{A}_n$.

\begin{Proposition}
Given $u,v\in U(\gg)$, $f,g\in S(\gg)$ the following formulas hold: $\xi(f)\triangleright g = f\star g$, $f\triangleright g = f g$, $\partial^\mu\triangleright f = \partial^\mu(f)$, $f\blacktriangleright_\gg 1 = \xi(f)$, $u\triangleright 1 = \xi^{-1}(u)$, $u\blacktriangleright_\gg v = u v$, $\partial^\mu\blacktriangleright_\gg(\hx_\nu v)= \delta^\mu_\nu v + \hx_\nu (\partial^\mu\blacktriangleright_\gg v)$, $\hy_\mu\blacktriangleright_\gg u = u \hx_\mu$ and more generally $\beta(v)\blacktriangleright_\gg u = u v$. The last two formulas imply
\begin{gather}\label{eq:betastar}
\hy_\mu\triangleright f = f\star x_\mu,\qquad \beta(\xi(g))\triangleright f = f\star g,\qquad f,g\in S(\gg).
\end{gather}
\end{Proposition}
\begin{proof} These properties are at length discussed and proved in our articles~\cite{scopr,halg,heisd}. The actions $\blacktriangleright_\gg$ and $\triangleright$ are there denoted $\blacktriangleright$ and $\triangleright$ respectively.
\end{proof}

\begin{Lemma}\label{lem:IUg}
The right ideal $I_{U(\gg)}$ $($compare~\eqref{eq:IA}$)$ is generated by $\hat{x}_\rho\otimes 1 - \mathcal{O}^\tau_\rho\otimes\hat{x}_\tau$, $\rho = 1,\dots, n$.
\end{Lemma}

\begin{proof}$I_{U(\gg)}$ is generated by $\beta(u)\otimes 1 - 1\otimes\alpha(u)$ where $u\in U(\gg)$ and it is enough to confine to the generators $u=\hx_\mu$ of $U(\gg)$. Then $\beta(\hx_\mu)\otimes 1 - 1\otimes \alpha(\hx_\mu) = \hy_\mu\otimes 1 - 1\otimes\hx_\mu
= \hx_\sigma \mathcal{O}^{-1\sigma}_\mu\otimes 1- 1\otimes\hx_\mu
= (\hx_\sigma\otimes 1 - \mathcal{O}^\rho_\sigma\otimes\hx_\rho)(\mathcal{O}^{-1\sigma}_\mu\otimes 1)$. Therefore $\beta(\hx_\mu)\otimes 1 - 1\otimes \alpha(\hx_\mu)$ is in the right ideal generated by elements $\hx_\sigma\otimes 1 - \mathcal{O}^\rho_\sigma\otimes\hx_\rho$. Conversely, $\hx_\sigma\otimes 1 - \mathcal{O}^\rho_\sigma\otimes\hx_\rho = (\hy_\mu\otimes 1 - 1\otimes\hx_\mu)(\mathcal{O}^\mu_\sigma\otimes 1)
\in I_{U(\gg)}$.
\end{proof}

\begin{Theorem}[{\cite[Theorem~4]{halg}}]\label{th:th4halg}
For the topological Heisenberg double $H = H_\gg = U(\gg)\hat\sharp\hat{S}(\gg)$ of the universal enveloping algebra $U(\gg)$, considered as a bialgebroid over $A = U(\gg)$, the right ideal $I^{(k)}$ which is the kernel of the canonical projection $H\otimes\cdots\otimes H\to H\otimes_A\cdots \cdots\otimes_A H$ $(k$~factors, all tensor products properly completed$)$ coincides with the right ideal $I'^{(k)}$ consisting of all elements $r$ in $H\otimes\cdots\otimes H$ such that
\begin{gather*}
\mu_{k-1} \left( r (\blacktriangleright_\gg\otimes\cdots\otimes\blacktriangleright_\gg)(a_1\otimes\cdots \otimes a_k)\right)= 0\qquad
\text{for all}\quad a_1,\dots,a_k\in H.
\end{gather*}
\end{Theorem}
For an arbitrary bialgebroid one can easily show that $I^{(k)}\subset I'^{(k)}$, but the converse inclusion is not present in general. \begin{Corollary}\label{cor:FDFimpliesDt}
Let $H = H_\gg$. For any $\mathcal{F}\in H\hat\otimes_{S(\gg)} H$ satisfying $\mathcal{F} I_0 = I_{U(\gg)}$ and having an inverse $\mathcal{F}^{-1}\in H\hat\otimes_{U(\gg)} H$, the satisfaction of the identity $(\Delta_H)_{\mathcal{F}}(h) = \mathcal{F}^{-1}\Delta_H(h)\mathcal{F} + I_{U(\gg)}$ for all $h\in H$ implies the cocycle condition~\eqref{eq:2coc}. In particular, this holds for the undeformed case, where $H = \hat{A}_n$ and the coproduct is $\Delta_0$ as in~\eqref{eq:Delta0}.
\end{Corollary}

\begin{proof}This follows from the coassociativity $\Delta_H$ by the reasoning at the end of the Remark~\ref{rem:cocyccoass} using that the nondegeneracy Theorem~\ref{th:th4halg} holds in the deformed case. Of course, the nondegeneracy is well known in the undeformed case. For example, if $k=2$, then up to the issue of completions, it boils down to the basic fact that the Weyl algebra as defined by generators and relations is {\em faithfully} represented on the space of polynomials.
\end{proof}

In this paper, we need the corollary in the undeformed case, as we study below the twists $\mathcal{F}_l$,~$\mathcal{F}_r$ on the undeformed $H = \hat{A}_n$. If~$\mathcal{F}$ is a candidate for a twist on undeformed~$H$ then the cocycle condition for $\mathcal{F}^{-1}$ (as a twist on deformed $H_\gg$) may be derived using Theorem~\ref{th:th4halg} in the nontrivial deformed case. This is useful as the cocycle condition for some examples of~$\mathcal{F}$ or~$\mathcal{F}^{-1}$ is studied in the literature without previous knowledge of the cocycle condition for~$\mathcal{F}$.

In~\cite{scopr}, we proved another formula for a twist deforming the coproduct $\Delta_0$ on $\hat{A}_n$ into the coproduct for $H_\gg$ (although without bialgebroid formalism) which is however less explicit:
\begin{gather}\label{eq:Fc}
\mathcal{F}_c = \sum_{i_1,\dots,i_n = 0}^\infty \frac{x_1^{i_1}\cdots x_n^{i_n}\otimes 1}{i_1!\cdots i_n!}\prod_{\nu = 1}^n\big[\big(\Delta_{\hat{S}(\gg)} - \Delta_0\big)(\partial^\nu)\big]^{i_\nu}
\end{gather}
as it is a series, where each summand involves $\Delta_{\hat{S}(\gg)}$ which is itself a series and not very explicit. Symbolically, using the normal ordering ($x$-s to the left, $\partial$-s to the right) operation ${:} \ {:}$ one can write this formula as the normally ordered exponential \begin{gather*}
\mathcal{F}_c = {:}\exp\left(\sum_{\mu = 1}^n (x_\mu\otimes 1)\big(\Delta_{\hat{S}(\gg^*)}-\Delta_0\big)(\partial^\mu)\right){:}.
\end{gather*}

\begin{Remark}[warning on variants of completions]\label{rem:wcompl}
The completions in~\cite{halg} are, roughly speaking, with respect to the cofiltrations on $U(\gg)^*$ and $\hat{S}(\gg^*)$ induced by duality from the standard filtrations on $U(\gg)$ and $S(\gg)$. This is {\em essentially different} from using the additional formal variable $h$ and $h$-adic completions as in Xu's work~\cite{xu}. Naively, to fit with his work, the Lie algebra generators (or equivalently the structure constants)
should be simply rescaled by the formal variable~$h$. For many simple purposes this gives an equivalent treatment to ours. The set of series which formally converge in two variants differs however. For the main results in the present article this is important. Namely, the twists $\mathcal{F}_l$, $\mathcal{F}_r$ in Section~\ref{s:newtwist} exist in both completions, the formulas as products of two exponentials make sense in our formalism, but these individual exponential factors do not exist in the $h$-adic completion, even after rescaling. Indeed, $\exp(\partial^\alpha\otimes x_\alpha)$ does not involve a small parameter (the reason is that $\partial^\alpha$ and $x_\alpha$ are dual and no rescaling could make their tensor product small) and is in fact related to an infinite-dimensional version of the canonical element, while (due cancellations in the expansion) the entire twist~$\mathcal{F}_l$ equals~$1$ plus a series of corrections involving the small parameter. Thus, unlike the exponential factors, the final result~$\mathcal{F}_l$ does exist in both formalisms and hence can be interpreted as defining a~deformation quantization in the sense used in Theorem~\ref{thm:defqinterp}.
\end{Remark}
\begin{Remark}The smash product algebra $U(\gg)\sharp_{{\boldsymbol \phi}}\hat{S}(\gg^*)$ can be equipped with another bialgebroid structure, namely over the commutative base algebra~$\hat{S}(\gg^*)$ and with the comultiplication~$\Delta'$ for which $\Delta'(u\sharp P) = (u_{(1)}\sharp 1)\otimes_{\hat{S}(\gg^*)}(u_{(2)}\sharp P)$ for $u\in U(\gg)$ and $P\in\hat{S}(\gg^*)$, where $u\mapsto u_{(1)}\otimes u_{(2)}$ denotes the (standard) cocommutative coproduct of $U(\gg)$. While our co\-pro\-duct~$\Delta$ takes values in a completed product over the noncommutative base, the coproduct $\Delta'$ is algebraic, taking values in the ordinary tensor product over the commutative base. Such scalar extensions over a~commutative base were known much before \cite[pp.~117--118]{Sweedler}. This bialgebroid does not fit the physical interpretation which we intended. The coproduct on~$\hat{S}(\gg^*)$ in our bialgebroid structure in this paper is dual to the algebra structure on $U(\gg)$ manifest in both algebroids and the coproduct on~$U(\gg)$ in the other bialgebroid structure is dual to the algebra structure on $\hat{S}(\gg^*)$ manifest in both algebroids. The two entire bialgebroids are mutually not dual in some sense standard for bialgebroids, although their factors $U(\gg)$ and $\hat{S}(\gg^*)$ are involved in (topological) Hopf algebra dualities.
\end{Remark}

\section{New twist}\label{s:newtwist}

In this section, we show that the two expressions $\mathcal{F}_l$ and $\mathcal{F}_r$ define (the same) twist of a completed Heisenberg double and that the twisted bialgebroid is $H_\gg$ from the previous sections.
\begin{Theorem}\label{thm:expadDelta}\quad
\begin{enumerate}\itemsep=0pt
\item[$(i)$] In symmetric ordering, the deformed coproduct
$\Delta_{\hat{S}(\gg^*)}$ on $\hat{S}(\gg^*)\cong U(\gg)^*$ is given by
\begin{gather*}
\Delta_{\hat{S}(\gg^*)}\partial^\mu = 1\otimes\partial^\mu + \partial^\alpha\otimes[\partial^\mu,\hx_\alpha] +\tfrac{1}{2} \partial^\alpha\partial^\beta\otimes [[\partial^\mu,\hx_\alpha],\hx_\beta] +\cdots,
\end{gather*}
where the right-hand side, though an element of~$\hat{S}(\gg^*)\hat\otimes\hat{S}(\gg^*)$, is calculated in the bigger algebra $H\hat\otimes H$
$($the completed tensor product over the ground field$)$. The same equation written in a symbolic form is
\begin{gather}\label{eq:expad1}
\Delta_{\hat{S}(\gg^*)}\partial^\mu = \exp(\partial^\alpha\otimes \operatorname{ad} (-\hx_\alpha))
(1\otimes\partial^\mu) = \exp(\operatorname{ad}(-\partial^\alpha\otimes \hx_\alpha)) (1\otimes\partial^\mu).
\end{gather}
$($The rightmost equality follows by noting that $[\partial^\alpha,1] = 0$.$)$

\item[$(ii)$] More generally, for $P\in \hat{S}(\gg^*)$ we have
\begin{gather*}
\Delta_{\hat{S}(\gg^*)} (P) = \exp(\partial^\alpha\otimes \operatorname{ad} (-\hx_\alpha))(1\otimes P) = \exp(\operatorname{ad}(-\partial^\alpha\otimes \hx_\alpha)) (1\otimes P).
\end{gather*}
\end{enumerate}
\end{Theorem}

\begin{proof} Part (i) is proven in~\cite{scopr} (and used in~\cite{heisd}).

(ii) To extend the identity to the formal power series just notice that the conjugation is a~homomorphism of completed algebras $\hat{H}\otimes\hat{H}\to\hat{H}$(the usual non-completed tensor product of completions). More explicitly, the deformed coproduct of the $n$-th order monomial in $\partial$-s has all summands in order $n$ or higher. If in a formal series we replace each monomial with a formal series with
nondecreased minimal order then we obtain only finitely many summands of each finite order hence again a formal series.
\end{proof}

\begin{Corollary}\label{cor:expady}
In symmetric ordering, the deformed coproduct $\Delta_{\hat{S}(\gg^*)}$ on
$\hat{S}(\gg^*)\cong U(\gg)^*$ is also given by
\begin{gather}\label{eq:expad2}
\Delta_{\hat{S}(\gg^*)}(\partial^\mu) = \exp(\mathrm{ad}(\hy_\alpha\otimes\partial^\alpha))
(\partial^\mu\otimes 1)
\end{gather}
and more generally $\Delta_{\hat{S}(\gg^*)}(P) = \exp(\mathrm{ad}(\hy_\alpha\otimes\partial^\alpha))(P\otimes 1)$ for all $P\in\hat{S}(\gg^*)$.
\end{Corollary}
\begin{proof} It is known~\cite{scopr} that the deformed coproduct $\Delta_{\hat{S}(\gg^*)}(\partial^\mu)\in \hat{S}(\gg^*)\hat\otimes\hat{S}(\gg^*)$ has the same symmetries as the Hausdorff formula $H(Z,W) = -H(-W,-Z)$; thus if we interchange the tensor factors in the coproduct and multiply each partial derivative with~$-1$ then we get the same coproduct with the overall minus sign. This rule of course applies only when the coproduct is written manifestly in terms of tensor products of formal power series in partial derivatives. Changing the sign of each partial derivative in the realizations~$\hat{x}_\alpha$ manifest in~(\ref{eq:expad1}) is equivalent to the change of~$\phi$ to~$\tilde\phi$ (because partial derivatives are contracted to the structure constants so it is the same as change of the sign in the structure constants), that is $\hat{x}_\alpha$ to $\hat{y}_\alpha$. However, the commutators like $[\partial^\mu,\hx_\alpha]$ when calculated have a different number of partials in each monomial than the raw expression. Indeed, one partial derivative and one power of $x$ (recall that $\hx_\rho = x_\sigma\phi^\sigma_\rho$ where $\phi^\sigma_\rho$ involves partial derivatives) drop out when commuting; similarly in each of the further commutators one power of~$x$ and one of the partial derivatives within the $\tilde\phi$-s drops out. As the number of commutators in each of the left tensor factors in~(\ref{eq:expad1}) is by one less than the number of partial derivatives in the raw tensor product expression, together with overall minus sign it amounts to no change in sign besides the change accounted in $\hx_\mu\mapsto\hy_\mu$. Thus
\begin{gather*} \Delta_{\hat{S}(\gg^*)} \partial^\mu = \partial^\mu\otimes 1 +
[\partial^\mu,\hat{y}_\alpha]\otimes \partial^\alpha + \frac{1}{2}
[[\partial^\mu,\hat{y}_\alpha],\hat{y}_\beta]\otimes \partial^\alpha\partial^\beta + \cdots.
\end{gather*}
The proof that the result extends to general $P\in\hat{S}(\gg^*)$ is completely analogous to the proof of part (ii) in Theorem~\ref{thm:expadDelta}.
\end{proof}

\begin{Definition} Denote
\begin{gather}\label{eq:eetwist}
\mathcal{F}_l = \exp(-\partial^\rho\otimes x_\rho)\exp(\partial^\sigma\otimes\hx_\sigma),\\
\label{eq:eertwist}
\mathcal{F}_R = \exp(-x_\rho\otimes\partial^\rho) \exp(\hy_\sigma\otimes\partial^\sigma)
\end{gather}
understood as elements in $H\hat\otimes_{S(\gg)} H$ with $\mathcal{F}^{-1}_L, \mathcal{F}^{-1}_R\in H\hat\otimes_{U(\gg)}H$. The right-hand sides of the formulas~(\ref{eq:eetwist}), (\ref{eq:eertwist}) understood as elements in $H\hat\otimes H$ define $\tilde{\mathcal{F}}_l,\tilde{\mathcal{F}}_r\in H\hat\otimes H$ with explicit inverses $\tilde{\mathcal{F}}_l^{-1} = \exp(-\partial^\sigma\otimes\hx_\sigma)\exp(\partial^\rho\otimes x_\rho)$, $\tilde{\mathcal{F}}_r^{-1} = \exp(-\hy_\sigma\otimes\partial^\sigma)\exp(x_\rho\otimes\partial^\rho)$.
\end{Definition}

Clearly, $\mathcal{F}_l = \tilde{\mathcal{F}}_l+I_0$, $\mathcal{F}_r = \tilde{\mathcal{F}}_r+I_0$. We define also $\mathcal{F}_l^{-1} := \tilde{\mathcal{F}}_l^{-1}+I_{U(\gg)}$, $\mathcal{F}_r^{-1} := \tilde{\mathcal{F}}_r^{-1}+I_{U(\gg)}$ (see~(\ref{eq:IA})), but at this point we can not yet claim that they are twist inverses in the sense of Remark~(\ref{rem:invcoc}) as the expressions $\mathcal{F}_l\mathcal{F}_l^{-1}$ etc.\ should be well defined what is proven only below.

\begin{Lemma}\label{lem:DeltaFd}
$\Delta_{\hat{S}(\gg^*)}\partial^\mu = \tilde{\mathcal{F}}_l^{-1}\Delta_0(\partial^\mu)\tilde{\mathcal{F}}_l = \tilde{\mathcal{F}}_r^{-1}\Delta_0(\partial^\mu)\tilde{\mathcal{F}}_r$ for $\mu = 1,\dots, n$. More generally, for any $P\in\hat{S}(\gg^*)$ $($formal power series in $\partial^1,\dots,\partial^n)$, $\Delta_{\hat{S}}(\gg^*) P = \tilde{\mathcal{F}}_l^{-1}\Delta_0(P)\tilde{\mathcal{F}}_l = \tilde{\mathcal{F}}_r^{-1}\Delta_0(P)\tilde{\mathcal{F}}_r$.
\end{Lemma}

\begin{proof}
By the Hadamard's formula $\exp(A)B\exp(-A) = \exp(\ad A)(B)$, (\ref{eq:expad1}) and~(\ref{eq:expad2}) read
\begin{gather}\label{eq:sen1}
\Delta_{\hat{S}(\gg^*)}\partial^\mu = \exp(-\partial^\rho \otimes\hx_\rho)(1\otimes\partial^\mu)\exp(\partial^\sigma\otimes\hx_\sigma),\\
\label{eq:sen2}
\Delta_{\hat{S}(\gg^*)}\partial^\mu = \exp(-\hy_\rho\otimes\partial^\rho)(\partial^\mu\otimes 1)\exp(\hy_\sigma\otimes\partial^\sigma).
\end{gather}
In particular, in the undeformed case when $C_{\mu\nu}^\lambda = 0$ and $\hx_\alpha$, $x_\alpha$ and $\hy_\alpha$ coincide, we obtain the formulas for the undeformed coproduct $\Delta_0$ (which can also easily be checked directly)
\begin{gather}\label{eq:sen3}
\Delta_0\partial^\mu = \exp(-\partial^\alpha\otimes x_\alpha)(1\otimes\partial^\mu)\exp(\partial^\alpha\otimes x_\alpha),\\
\label{eq:sen4}
\Delta_0\partial^\mu = \exp(-x_\alpha\otimes\partial^\alpha)
(\partial^\mu\otimes 1)\exp(x_\alpha\otimes\partial^\alpha).
\end{gather}
Comparing the formulas for the deformed and for the undeformed case we obtain new formulas relating $\Delta_0$ to $\Delta_{\hat{S}(\gg^*)}(\gg^*)$. Indeed, comparing~(\ref{eq:sen1}) and (\ref{eq:sen3}) we obtain
\begin{gather*}
\Delta_{\hat{S}(\gg^*)}(\partial^\mu) = \tilde{\mathcal{F}}^{-1}_L\Delta_0(\partial^\mu)\tilde{\mathcal{F}}_l,
\end{gather*}
and similarly comparing (\ref{eq:sen2}) to (\ref{eq:sen4}) we obtain
\begin{gather*}
\Delta_{\hat{S}(\gg^*)}(\partial^\mu) = \tilde{\mathcal{F}}^{-1}_R\Delta_0(\partial^\mu)\tilde{\mathcal{F}}_r.
\end{gather*}
To extend the identities to the formal power series proceed as in the proof of part~(ii) to Theorem~\ref{thm:expadDelta}.

We would like to have the same identities in~$H\hat\otimes_{U(\gg)}H$, with $\mathcal{F}$ instead of $\tilde{\mathcal{F}}$, but for this the calculation should not depend on a representative, that is $\mathcal{F}^{-1}_L\cdot I_0 \subset I_{U(\gg)}$, which is proven below in Proposition~\ref{prop:ideal}.
\end{proof}

\begin{Definition}
Define the map of left $U(\gg)$-modules $\tilde\Delta\colon H\to H\hat\otimes H$ (completed tensor product over the field) by $\tilde\Delta(u\sharp P) = u\Delta_{\hat{S}(\gg^*)}(P)$ (the inclusion $\hat{S}(\gg^*)\hat\otimes\hat{S}(\gg^*)\hookrightarrow H\hat\otimes H$ understood).
\end{Definition}
\begin{Remark}\label{rem:tDDH}
Clearly $\tilde\Delta(u\sharp P)+I_{U(\gg)} = \Delta_H(u\sharp P)$, but unlike $\Delta_H$ and $\Delta_{\hat{S}(\gg^*)}$ the map $\tilde\Delta$ is not multiplicative, that is $\tilde\Delta(h)\tilde\Delta(h')\neq\tilde\Delta(h h')$ for general $h,h'\in H$. We use this map to compute more precisely with $\tilde{\mathcal{F}}_l, \tilde{\mathcal{F}}_r$.
\end{Remark}
\begin{Lemma}\label{lem:DeltaFx}\quad
\begin{enumerate}\itemsep=0pt
\item[$(i)$] For $\mu = 1,\dots, n$ we have
\begin{gather}
\tilde{\mathcal{F}}_l^{-1}(x_\mu\otimes 1)\tilde{\mathcal{F}}_l = \tilde\Delta(x_\mu)
+ \hx_\tau\big(\phi^{-1}\big)^\tau_\mu\otimes 1 - \big(\tilde\phi^{-1}\big)^\tau_\mu\otimes\hx_\tau\nonumber\\
\hphantom{\tilde{\mathcal{F}}_l^{-1}(x_\mu\otimes 1)\tilde{\mathcal{F}}_l =}{}
+ \big(\mathcal{O}^\sigma_\lambda\otimes\hx_\sigma-\hx_\lambda\otimes 1\big)\tilde\Delta\big(\big(\phi^{-1}\big)^\lambda_\mu),\label{eq:tFLx1FL}
\\
\label{eq:tFL1xFL}
\tilde{\mathcal{F}}_l^{-1}(1\otimes x_\mu)\tilde{\mathcal{F}}_l =
\tilde\Delta(x_\mu)+\big(\mathcal{O}^\sigma_\lambda\otimes\hx_\sigma-\hx_\lambda\otimes 1\big)\tilde\Delta\big(\big(\phi^{-1}\big)^\lambda_\mu\big),
\\ \label{eq:tFR1xFR}
\tilde{\mathcal{F}}_r^{-1}(1\otimes x_\mu)\tilde{\mathcal{F}}_r = \tilde\Delta(x_\mu)+1\otimes\hx_\mu - \hy_\tau\otimes\big(\phi^{-1}\big)^\tau_\mu,
\\ \label{eq:tFRx1FR}
\tilde{\mathcal{F}}_r^{-1}(x_\mu\otimes 1)\tilde{\mathcal{F}}_r = \tilde\Delta(x_\mu).
\end{gather}

\item[$(ii)$] $x_\mu\otimes 1 - \big(\tilde\phi^{-1}\big)^\tau_\mu\otimes\hat{x}_\tau\in I_{U(\gg)}$,
$\mathcal{O}^\sigma_\lambda\otimes\hx_\sigma-\hx_\lambda\otimes 1\in I_{U(\gg)}$
and $1\otimes\hx_\mu - \hy_\tau\otimes\big(\phi^{-1}\big)^\tau_\mu\in I_{U(\gg)}$.

\item[$(iii)$] $\Delta_H x_\mu = \tilde{\mathcal{F}}_l^{-1}(x_\mu\otimes 1)\tilde{\mathcal{F}}_l + I_{U(\gg)}
= \tilde{\mathcal{F}}_r^{-1}(1\otimes x_\mu)\tilde{\mathcal{F}}_r + I_{U(\gg)}$,
where $I_{U(\gg)}$ is the right ideal generated by $\beta(u)\otimes 1 - 1\otimes \alpha(u)$ for all $u\in U(\gg)$.
\end{enumerate}
\end{Lemma}

\begin{proof} We first compute $\tilde{\mathcal{F}}^{-1}_l(x_\mu\otimes 1)\tilde{\mathcal{F}}_l$ and $\tilde{\mathcal{F}}^{-1}_l(1\otimes x_\mu)\tilde{\mathcal{F}}_l$. In order to conjugate with $\tilde{\mathcal{F}}^{-1}_L$ which is a product of exponentials, we first conjugate with the ``inner'' exponential
\begin{gather}\label{eq:ex1e}
\exp(\partial^\rho\otimes x_\rho)(x_\mu\otimes 1)\exp(-\partial^\sigma\otimes x_\sigma) = x_\mu\otimes 1 + 1\otimes x_\mu,\\
\label{eq:e1xe}
\exp(\partial^\rho\otimes x_\rho)(1\otimes x_\mu)\exp(-\partial^\sigma\otimes x_\sigma) = 1\otimes x_\mu.
\end{gather}
Now we need to apply outer exponentials.

By induction on $k = 0,1,2,\ldots$ one checks that
(in notation~(\ref{CCmatrix}))
\begin{gather}\label{eq:adC}
\ad^k(\partial^\rho\otimes\hx_\rho) (1\otimes\hx_\mu) = [(-\CC)^k]^\tau_\mu\otimes\hx_\tau.
\end{gather}
Hadamard's formula, identity $\ad(\partial^\rho\otimes\hx_\rho)(x_\mu\otimes 1)=1\otimes\hx_\mu$ and equation~(\ref{eq:adC}) imply
\begin{gather*}
\exp(-\partial^\sigma\otimes\hx_\sigma)(x_\mu\otimes 1)\exp(\partial^\rho\otimes\hat{x}_\rho) = x_\mu\otimes 1 - \sum_{k=1}^\infty \frac{(\CC^{k-1})^\tau_\mu}{k!}\otimes\hx_\tau\\
\hphantom{\exp(-\partial^\sigma\otimes\hx_\sigma)(x_\mu\otimes 1)\exp(\partial^\rho\otimes\hat{x}_\rho) }{}
 = x_\mu\otimes 1 - \left(\frac{e^{\CC}-1}{\CC}\right)^\tau_\mu\otimes\hx_\tau,
\end{gather*}
therefore by~(\ref{eq:phi})
\begin{gather}\label{eq:eFe}
\exp(-\partial^\sigma\otimes\hx_\sigma)(x_\mu\otimes 1)\exp(\partial^\rho\otimes
\hat{x}_\rho) = x_\mu\otimes 1 - \big(\tilde\phi^{-1}\big)^\tau_\mu\otimes\hat{x}_\tau.
\end{gather}
Now conjugate $1\otimes x_\mu = 1\otimes\hy_\tau\OO^\tau_\sigma(\phi^{-1})^\sigma_\mu$ with $\exp(-\partial^\sigma\otimes\hx_\sigma)$, using
\begin{gather*}1\otimes\hy_\tau\OO^\tau_\sigma\big(\phi^{-1}\big)^\sigma_\mu = (1\otimes\hy_\tau)\exp(\partial^\nu\otimes\hx_\nu)\exp(-\partial^\lambda\otimes\hx_\lambda)\big(1\otimes\OO^\tau_\sigma \big(\phi^{-1}\big)^\sigma_\mu\big)\end{gather*}
and $[\hx_\sigma,\hy_\tau] = 0$, therefore obtaining
\begin{gather}
\exp(-\partial^\sigma\otimes\hx_\sigma)\big(1\otimes \hy_\tau\OO^\tau_\lambda\big(\phi^{-1}\big)^\lambda_\mu\big)\exp(\partial^\rho\otimes\hx_\rho) \nonumber\\
\qquad{} = (1\otimes\hy_\tau)\Delta_{\hat{S}(\gg^*)}(\OO^\tau_\lambda)\Delta_{\hat{S}(\gg^*)}\big(\big(\phi^{-1}\big)^\lambda_\mu\big)
 = (\mathcal{O}^\sigma_\lambda\otimes\hx_\sigma)\Delta_{\hat{S}(\gg^*)}\big(\big(\phi^{-1}\big)^\lambda_\mu\big),\nonumber\\
\label{eq:eSe}
\exp(-\partial^\sigma\otimes\hx_\sigma)\big(1\otimes \hy_\tau\OO^\tau_\lambda\big(\phi^{-1}\big)^\lambda_\mu\big)\exp(\partial^\rho\otimes\hx_\rho)\nonumber\\
\qquad{}
= \tilde\Delta(x_\mu) + (\mathcal{O}^\sigma_\lambda\otimes\hx_\sigma-\hx_\lambda\otimes 1)\tilde\Delta\big(\big(\phi^{-1}\big)^\lambda_\mu\big),
\end{gather}
where we also used the known fact \cite{halg} that $\Delta_{S(\gg^*)}(\mathcal{O}^\tau_\lambda) = \mathcal{O}_\lambda^\sigma\otimes\mathcal{O}^\tau_\sigma$, Theorem~\ref{thm:expadDelta}(ii) for $P = \OO^\tau_\sigma$ and $P = (\phi^{-1})^\sigma_\mu$ and the fact that $(\Delta_H)|_{\hat{S}(\gg^*)}$ agrees with (deformed) $\Delta_{\hat{S}(\gg^*)}$ followed by the inclusion into $H\hat\otimes_{U(\gg)}H$. According to~(\ref{eq:ex1e}) we obtain~(\ref{eq:tFLx1FL}) by adding~(\ref{eq:eFe}) and~(\ref{eq:eSe}) and similarly according to~(\ref{eq:e1xe}) we obtain~(\ref{eq:tFL1xFL}) from~(\ref{eq:eSe}).

To compute $\tilde{\mathcal{F}}^{-1}_r(1\otimes x_\mu)\tilde{\mathcal{F}}_r$ and $\tilde{\mathcal{F}}^{-1}_r(x_\mu\otimes 1)\tilde{\mathcal{F}}_r$ we firstly conjugate with $\exp(x_\rho\otimes\partial^\rho)$,
\begin{gather}
\exp(x_\rho\otimes\partial^\rho)(1\otimes x_\mu) \exp(-x_\sigma\otimes\partial^\sigma) = 1\otimes x_\mu + x_\mu \otimes 1,\nonumber\\
\exp(x_\rho\otimes\partial^\rho)(x_\mu\otimes 1)\exp(-x_\sigma\otimes\partial^\sigma) = x_\mu \otimes 1.\label{eq:ex1eR}
\end{gather}
Then we need to conjugate each summand at the right-hand side with the outer exponential $\exp(-\hy_\sigma\otimes\partial^\sigma)$. By induction on $k = 1,2,\ldots$ one shows the analogue of~(\ref{eq:adC}) (or just notice that $\hx\mapsto \hy$ by changing the sign of the structure constants and interchange the tensor factors):
\begin{gather*}
\ad^k(\hy_\rho\otimes\partial^\rho) (\hy_\mu\otimes 1) = \hy_\tau\otimes\big(\CC^k\big)^\tau_\mu.
\end{gather*}
Along with Hadamard's formula and
$\ad(\hy_\rho\otimes\partial^\rho)(1\otimes x_\mu)=1\otimes\hy_\mu$
this implies
\begin{gather*}
\begin{split}
&\exp(-\hy_\sigma\otimes\partial^\sigma)(1\otimes x_\mu)\exp(\hy_\rho\otimes\partial^\rho) = 1\otimes x_\mu - \sum_{k=1}^\infty \hy_\tau\otimes\frac{\big((-\CC)^{k-1}\big)^\tau_\mu}{k!}\otimes\hx_\tau \\
& \hphantom{\exp(-\hy_\sigma\otimes\partial^\sigma)(1\otimes x_\mu)\exp(\hy_\rho\otimes\partial^\rho)}{}
= 1\otimes x_\mu -\hy_\tau\otimes\left(\frac{e^{-\CC}-1}{-\CC}\right)^\tau_\mu,
\end{split}
\end{gather*}
therefore by~(\ref{eq:phi})
\begin{gather}\label{eq:eFeR}
\exp(-\hy_\sigma\otimes\partial^\sigma)(1\otimes x_\mu)\exp(\hy_\rho\otimes\partial^\rho) = 1\otimes x_\mu - \hat{y}_\tau\otimes\big(\phi^{-1}\big)^\tau_\mu.
\end{gather}
Now we conjugate
the second summand $x_\mu\otimes 1 = \hx_\lambda(\phi^{-1})^\lambda_\mu\otimes 1$ on the right-hand side of~(\ref{eq:ex1eR}) with $\exp(-\hy_\sigma\otimes\partial^\sigma)$, using $[\hx_\sigma,\hy_\tau] = 0$, therefore obtaining
\begin{gather}\label{eq:eSeR}
\exp(-\hy_\sigma\otimes\partial^\sigma)\big(\hx_\lambda\big(\phi^{-1}\big)^\lambda_\mu\otimes 1\big)\exp(\hy_\rho\otimes\partial^\rho) = (\hx_\lambda\otimes_{U(\gg)}1)\Delta_{\hat{S}(\gg^*)}\big(\big(\phi^{-1}\big)^\lambda_\mu\big) = \tilde\Delta(x_\mu),\!\!\!
\end{gather}
where we also used Corollary~\ref{cor:expady} for $P = \big(\phi^{-1}\big)^\sigma_\mu$. Adding~(\ref{eq:eFeR}) and~(\ref{eq:eSeR}) we obtain~(\ref{eq:tFR1xFR}) and similarly~(\ref{eq:eSeR}) alone gives~(\ref{eq:tFRx1FR}).

(ii) The formula~(\ref{eq:xyphi}) gives $x_\mu = \hat{x}_\sigma \big(\phi^{-1}\big)^\sigma_\mu$ and~(\ref{eq:yeqxO}) gives $\tilde\phi^{-1} = \mathcal{O}\phi^{-1}$, while the right ideal~$I_{U(\gg)}$ is by Lemma~\ref{lem:IUg} generated by elements of the form $\hat{x}_\rho\otimes 1 - \mathcal{O}^\tau_\rho\otimes\hat{x}_\tau$, hence
\begin{gather*}
x_\mu\otimes 1 - \big(\tilde\phi^{-1}\big)^\tau_\mu\otimes\hat{x}_\tau = (\hat{x}_\rho\otimes 1 - \OO^\tau_\rho\otimes\hat{x}_\tau)\big(\big(\phi^{-1}\big)^\rho_\mu\otimes 1\big)\in I_{U(\gg)}.\end{gather*}

Similarly, $1\otimes x_\mu - \hat{y}_\tau\otimes\big(\phi^{-1}\big)^\tau_\mu = (1\otimes \hat{x}_\tau - \hat{y}_\tau\otimes 1)\big(1\otimes\big(\phi^{-1}\big)^\tau_\mu\big)\in I_{U(\gg)}$.

Part (iii) follows from (i) and (ii) (note Remark~\ref{rem:tDDH}).
\end{proof}

Finally, one proves that the undeformed right ideal $I_0$ generated by $x_\mu\otimes 1 - 1\otimes x_\mu$ after twist ends in the deformed right ideal $I_{U(\gg)}$. In fact,
\begin{Proposition}\label{prop:ideal}
\begin{gather*}
\tilde{\mathcal{F}}_l^{-1}(x_\mu\otimes 1-1\otimes x_\mu)\tilde{\mathcal{F}}_l =
\hat{x}_\tau\big(\phi^{-1}\big)^\tau_\mu\otimes 1 - \big(\tilde\phi^{-1}\big)^\tau_\mu\otimes\hat{x}_\tau\in I_{U(\gg)},
\\
\tilde{\mathcal{F}}_r^{-1}(x_\mu\otimes 1-1\otimes x_\mu)
\tilde{\mathcal{F}}_r =\hy_\tau\otimes\big(\phi^{-1}\big)^\tau_\mu- 1\otimes\hx_\mu \in I_{U(\gg)},
\\
\tilde{\mathcal{F}}_l^{-1}(x_\mu\otimes 1-1\otimes x_\mu)\in I_{U(\gg)},\qquad
\tilde{\mathcal{F}}_r^{-1}(x_\mu\otimes 1-1\otimes x_\mu)\in I_{U(\gg)},
\\
\mathcal{F}_l^{-1}(x_\mu\otimes 1-1\otimes x_\mu) =
\mathcal{F}_r^{-1}(x_\mu\otimes 1-1\otimes x_\mu) = 0 + I_{U(\gg)}
\in H\hat\otimes_{U(\gg)} H,
\\
\tilde{\mathcal{F}}_l^{-1}\cdot I_0 = I_{U(\gg)} = \tilde{\mathcal{F}}_R^{-1}\cdot I_0,
\qquad
\tilde{\mathcal{F}}_l\cdot I_{U(\gg)} = I_0 = \tilde{\mathcal{F}}_R \cdot I_{U(\gg)},
\\
\mathcal{F}_l\mathcal{F}_l^{-1} = 1\otimes 1 + I_0,\qquad
\mathcal{F}_l^{-1}\mathcal{F}_l = 1\otimes 1 + I_{U(\gg)},
\\
\mathcal{F}_r\mathcal{F}_r^{-1} = 1\otimes 1 + I_0,\qquad
\mathcal{F}_r^{-1}\mathcal{F}_r = 1\otimes 1 + I_{U(\gg)}.
\end{gather*}
\end{Proposition}
\begin{proof} The first line follows by subtracting~(\ref{eq:tFL1xFL}) from~(\ref{eq:tFLx1FL}) and the second line by subtrac\-ting~(\ref{eq:tFR1xFR}) from~(\ref{eq:tFRx1FR}). The third line follows by multiplying the first by $\tilde{\mathcal{F}}_r^{-1}$ and multiplying the second line by $\tilde{\mathcal{F}}_r^{-1}$. The fourth line is clearly just the restatement of the third in the quotient.

$\tilde{\mathcal{F}}_l^{-1} I_0\subset I_{U(\gg)}$ follows from the first line by noticing that the elements $x_\mu\otimes 1-1\otimes x_\mu$, $\mu = 1,\dots,n$ generate $I_0$ and similarly $\tilde{\mathcal{F}}_r^{-1} I_0\subset I_{U(\gg)}$ follows from the second line.

To show $\tilde{\mathcal{F}}_l I_{U(\gg)} \subset I_0$ multiply the first line with $\tilde{\mathcal{F}_l}$ from the left to obtain
\begin{gather*}
\tilde{\mathcal{F}_l}\big(x_\mu\otimes 1 - \big(\tilde\phi^{-1}\big)^\tau_\mu\otimes\hat{x}_\tau\big) = (x_\mu\otimes 1 - 1\otimes x_\mu)\cdot\tilde{\mathcal{F}}_L\in I_{U(\gg)},
\end{gather*}
and note that this is sufficient because the elements of the form $x_\tau\otimes 1 - (\tilde\phi^{-1})^\tau_\mu\otimes\hat{x}_\tau$ genera\-te~$I_{U(\gg)}$ because $(x_\mu\otimes 1 - (\tilde\phi^{-1})^\tau_\mu\otimes\hat{x}_\tau)(\tilde\phi^\mu_\nu\otimes 1) = \hat{y}_\nu\otimes 1 - 1\otimes\hat{x}_\nu$. Similarly we multiply the second line with $\tilde{\mathcal{F}}_R$ from the left to obtain
\begin{gather*}
\tilde{\mathcal{F}_R}\big(\hy_\tau\otimes\big(\phi^{-1}\big)^\tau_\mu- 1\otimes\hx_\mu\big) = (x_\mu\otimes 1 - 1\otimes x_\mu)\cdot\tilde{\mathcal{F}}_R^{-1}\in I_0.
\end{gather*}
This is sufficient to conclude $\tilde{\mathcal{F}}_R I_{U(\gg)}\subset I_0$ after observing that the elements of the form $\hy_\tau\otimes\big(\phi^{-1}\big)^\tau_\mu- 1\otimes\hx_\mu$ generate $I_{U(\gg)}$, indeed $\big(\hy_\tau\otimes\big(\phi^{-1}\big)^\tau_\mu- 1\otimes\hx_\mu\big) (1\otimes\phi^\tau_\nu\otimes 1) = \hy_\nu\otimes 1-1\otimes\hx_\nu$. Now $\tilde{\mathcal{F}}^{-1} I_0\subset I_{U(\gg)}$ and $\tilde{\mathcal{F}}I_{U(\gg)}\subset I_0$ together imply the equality whenever $\tilde{\mathcal{F}}^{-1}$ and $\tilde{\mathcal{F}}$ are strict inverses in $H\hat\otimes H$.

The assertions on $\mathcal{F}_l$, $\mathcal{F}_r$ are the direct corollary: the products $\mathcal{F}_l\mathcal{F}_l^{-1}$, $\mathcal{F}_r\mathcal{F}_r^{-1}$, $\mathcal{F}^{-1}_l\mathcal{F}_l$, $\mathcal{F}^{-1}_r\mathcal{F}_r$ are well defined so we can compute the representatives using $\tilde{\mathcal{F}}_l$, $\tilde{\mathcal{F}}_r$, $\tilde{\mathcal{F}}^{-1}_l$, $\tilde{\mathcal{F}}_r^{-1}$.
\end{proof}

\begin{Proposition}\label{prop:FDFm}
For every $h\in H = H_\gg$,
\begin{gather*}
\Delta_H(h) = \mathcal{F}_l^{-1}\Delta_0(h)\mathcal{F}_l = \mathcal{F}_r^{-1}\Delta_0(h)\mathcal{F}_r.
\end{gather*}
\end{Proposition}

\begin{proof} By Proposition~\ref{prop:ideal} the expressions $\mathcal{F}_l^{-1}\Delta_0(h)\mathcal{F}_l$ and $\mathcal{F}_r^{-1}\Delta_0(h)\mathcal{F}_r$ are well defined (do not depend on the representative of $\Delta_0(h)$). We can compute a representative of the resulting class modulo $I_{U(\gg)}$ as $\tilde{\mathcal{F}}^{-1}\Delta_0'(h)\tilde{\mathcal{F}}$ where $\Delta_0'(h)$ is any representative of $\Delta_0(h)$ and $\mathcal{F}$ is~$\mathcal{F}_l$ or~$\mathcal{F}_r$. This way the statement of the proposition for the generators $h = \partial^\mu$ follows by Lemma~\ref{lem:DeltaFd} and for the rest of generators $\hat{x}_\alpha$ by Lemma~\ref{lem:DeltaFx}. We need to extend the statement for all $h\in H$ by linearity and some sort of multiplicativity. Some care is however needed to achieve this.

Namely, $\sum_\alpha h_\alpha\otimes h'_\alpha \mapsto\tilde{\mathcal{F}}^{-1} \big(\sum_\alpha h_\alpha\otimes h'_\alpha\big) \tilde{\mathcal{F}}$ is a~homomorphism of algebras $H\hat\otimes H\to H\hat\otimes H$ and by $\mathcal{F}^{-1}\cdot I_0 = I_{U(\gg)}$ it induces a well defined map of vector spaces $\mathcal{F}^{-1} (-)\mathcal{F}\colon H\hat\otimes_{S(\gg)}H\to H\hat\otimes_{U(\gg)} H$. Regarding that $H\hat\otimes_{U(\gg)} H$ is not an algebra, we can not say that this map is a homomorphism of algebras. Suppose $h_1,h_2\in H$ are such that $\mathcal{F}^{-1}\Delta_0(h_i)\mathcal{F}=\Delta_H(h_i)$. Regarding that the images of $\Delta_H$ and $\Delta_0$ are algebras, that $\Delta_H, \Delta_0$ viewed as corestrictions to the images are multiplicative, and $\mathcal{F}\mathcal{F}^{-1} = 1\otimes 1 + I_0$, we calculate
\begin{gather*}\mathcal{F}^{-1}\Delta_0(h_1 h_2)\mathcal{F}=
\mathcal{F}^{-1}\Delta_0(h_1)\mathcal{F}\mathcal{F}^{-1}\Delta_0(h_2)\mathcal{F}=
\Delta_H(h_1)\Delta_H(h_2) = \Delta_H(h_1 h_2),\end{gather*}
which would be sufficient to end the proof. However, we freely used associativity and cancellations though the factors do not belong to an associative algebra. Associativity holds for the representatives in $H\hat\otimes H$, hence it is enough that all the products involved are well defined up to an appropriate right ideal. In our case we inspect this sequentially from the right to the left for the products involved, using $\mathcal{F} I_{U(\gg)} = I_0$, $\mathcal{F}^{-1} I_0 = I_{U(\gg)}$, $\Delta_0(h) I_0 \subset I_0$, $\Delta_H(h) I_{U(\gg)} \subset I_0$ and that $I_0$, $I_{U(\gg)}$ are right ideals.
\end{proof}

\begin{Lemma}[\cite{scopr}]\label{lem:exp}
In symmetric ordering, $\exp\big(\sum_\alpha t_\alpha x_\alpha\big)\blacktriangleright_\gg 1 = \exp\big(\sum_\alpha t_\alpha\hx_\alpha\big)$ for any formal variables $t_\alpha$ which commute with $x_\beta$, $\hx_\beta$. Conversely, $\exp\big(\sum_\alpha t_\alpha\hx_\alpha\big)\triangleright 1 = \exp\big(\sum_\alpha t_\alpha x_\alpha\big)$. In particular,
\begin{gather}
\exp\bigg(\sum_\alpha \partial^\alpha\otimes x_\alpha\bigg)(\cdot \otimes\blacktriangleright_\gg) (1\otimes 1) =
\exp\bigg(\sum_\alpha \partial^\alpha\otimes \hat{x}_\alpha\bigg),\nonumber\\
\label{eq:exp2exp}
\exp\bigg(\sum_\alpha \partial^\alpha\otimes \hat{x}_\alpha\bigg)
(\cdot \otimes\triangleright) (1\otimes 1) =\exp\bigg(\sum_\alpha \partial^\alpha\otimes x_\alpha\bigg).
\end{gather}
\end{Lemma}

\begin{Theorem} \label{th:main}
$\mathcal{F}_l$ and $\mathcal{F}_r$ are
counital Drinfeld twists for $S(\gg)$-bialgebroid
on completed Weyl algebra $\hat{A}_n$ and by twisting
they yield the Heisenberg double $H_\gg$
of the corresponding universal enveloping algebra $U(\gg)$
with its canonical $U(\gg)$-bialgebroid structure.
\end{Theorem}
\begin{proof} Both $\mathcal{F}_l$ and $\mathcal{F}_r$ are invertible.
Proposition~\ref{prop:FDFm} and Corollary~\ref{cor:FDFimpliesDt}
together imply that the Drinfeld cocycle condition~(\ref{eq:2coc}) holds.
We need to show the counitality. One has to be careful when checking this,
regarding that $\epsilon\colon H_\gg\to U(\gg)$ is not a homomorphism.
However, for the symmetric ordering, checking this is still not a difficult.
Recall from the axioms of the bialgebroid
that the undeformed counit is given by $\epsilon(h) = h\triangleright 1$. Thus
\begin{gather*}
(\epsilon\otimes 1)\mathcal{F}_l = \exp(-\partial^\rho\otimes x_\rho)\exp(\partial^\sigma\otimes\hx_\sigma)(\triangleright 1\otimes \cdot 1)
= 1\otimes 1,
\end{gather*}
because all the higher order terms have positive power of at least some $\partial^\alpha$-s thus yielding zero when acting upon $1$.
The second counitality condition is a bit more involved; using the fact that~$\triangleright$ extends the regular action of $U(\gg)$ on itself and~(\ref{eq:exp2exp}) we compute
\begin{align*}
(1\otimes \epsilon)\mathcal{F}_l &=\exp(-\partial^\rho\otimes x_\rho)\exp(\partial^\sigma\otimes\hx_\sigma)(\cdot 1\otimes \triangleright 1)\\
&= \exp(-\partial^\rho\otimes x_\rho)(\id\otimes \triangleright)\exp(\partial^\sigma\otimes x_\sigma)\\
&= \exp(-\partial^\rho\otimes x_\rho)(\id\otimes \triangleright)\exp(\partial^\sigma\otimes x_\sigma)(\cdot 1\otimes \triangleright 1)\\
&=(\exp(-\partial^\rho\otimes x_\rho)\exp(\partial^\sigma\otimes x_\sigma))(\cdot 1\otimes \triangleright 1)\\
&= 1\otimes 1.
\end{align*}
In the third line we used Lemma~\ref{lem:exp}. Similarly, one shows that $\mathcal{F}_r$ is counital.

The new base algebra is $S(\gg)$ with the $\mathcal{F}_l$-twisted (or $\mathcal{F}_r$-twisted) star product. We need to show that it is canonically isomorphic to $U(\gg)$ in the sense that
\begin{gather*}
\mu \mathcal{F}_l(\triangleright\otimes\triangleright)(g\otimes f) =\mu \mathcal{F}_r(\triangleright\otimes\triangleright)(g\otimes f)
= \xi^{-1}(\xi(g)\cdot_{U(\gg)}\xi(f)),\qquad g,f\in S(\gg),
\end{gather*}
where $\xi\colon S(\gg)\stackrel\cong\to U(\gg)$ is the symmetrization map. Since we know that $\mathcal{F}_l,\mathcal{F}_r$ are Drinfeld cocycles, we know that in both cases the corresponding star product is associative. Regarding that $\xi$ transports the product in $U(\gg)$ to the star product $g\star f := \xi^{-1}(\xi(g)\cdot_{U(\gg)}\xi(f)) = \xi(g)\triangleright f$ every element in $S(\gg)$ is a star polynomial in the generators in $\gg$. Therefore the associativity and the star products of the form $x_\mu\star f$ with general $f$ (or
alternatively, all star products of the form $g\star x_\mu$) are sufficient to determine the star product $g\star f$ for general $g$,~$f$. Thus for $\mathcal{F}_l$ it is sufficient to check that $\mu \mathcal{F}_l(\triangleright x_\mu\otimes \triangleright f) = \hat{x}_\mu\triangleright f = \hat{x}_\mu^\phi(f)$ for all $f$ and $\mu \mathcal{F}_r(\triangleright g\otimes \triangleright x_\mu) = \hat{y}_\mu\triangleright g = \hat{y}^\phi_\mu(g)$ for all $g$. When acting by $\triangleright$ we take the realization in the Weyl algebra and apply the corresponding differential operator. Thus all the higher derivatives drop out when acting on $x_\mu$ and we obtain $\mathcal{F}_l(\triangleright x_\mu\otimes \triangleright f) = x_\mu\otimes f - 1\otimes x_\mu f + 1\otimes\hat{x}^\phi_\mu(f)$. After applying $\mu$ on this equality, the first two summands cancel and we obtain $\hat{x}^\phi_\mu(f)$. Similarly, $\mathcal{F}_r(\triangleright g\otimes \triangleright x_\mu)= \hat{y}_\mu(g)\otimes 1 - g\otimes x_\mu + x_\mu g\otimes 1$ and after applying $\mu$ the second summand cancels with the third and we obtain $\mu\mathcal{F}_r(\triangleright g\otimes \triangleright x_\mu) =\hat{y}_\mu\triangleright g \stackrel{(\ref{eq:betastar})}= g\star x_\mu$ as required.

Finally, we need to compute the deformed source and target map from the twist and compare them with the known source and target map in the deformed case. Regarding that we already know that $\mathcal{F}_l$ and $\mathcal{F}_r$ are Drinfeld twists, Xu's theorem tells us that $\alpha_{\mathcal{F}_l}$, $\beta_{\mathcal{F}_l}$, $\alpha_{\mathcal{F}_r}$, $\beta_{\mathcal{F}_r}$ are algebra maps by the construction, it is sufficient to check the agreement with known deformed source and target maps on the algebra generators~$x_\mu$
\begin{gather*}
\alpha_{\mathcal{F}_l}(f)=\mu(\alpha_0\otimes\id)\exp(-\partial^\rho\otimes x_\rho)\exp(\partial^\sigma\otimes\hx_\sigma)(\triangleright f\otimes 1),
\end{gather*}
where $\alpha_0=\beta_0$ is the source map in the undeformed case, which is then equal to the undeformed target map. When $f = x_\mu$, after expanding the exponentials only three summands survive, with at most one partial derivative before applying $\triangleright$. Thus we obtain $\alpha_{\mathcal{F}_L}(x_\mu) = \alpha(x_\mu)-x_\mu+\hat{x}_\mu = \hat{x}_\mu = \alpha(\xi(x_\mu))$ as required. For $\beta$ we first interchange
the tensor factors in $\mathcal{F}_l$,
\begin{gather*}
\beta_{\mathcal{F}_l}(f)=\mu(\beta_0\otimes\id)\exp(-x_\rho\otimes\partial^\rho)\exp(\hx_\sigma\otimes\partial^\sigma)(\triangleright f\otimes 1).
\end{gather*}
Regarding that the second tensor factor commutes, we can compute $\exp(-x_\rho\otimes\partial^\rho)\exp(\hx_\sigma\otimes\partial^\sigma)(\triangleright x_\mu\otimes 1)$ as if $\partial$-s are independent formal variables. Now $\exp(k^\sigma\hx_\sigma)\triangleright x_\mu = \exp(k^\sigma x_\sigma)\star x_\mu =
\hy_\mu\triangleright\exp(k^\sigma x_\sigma)= x_\tau\tilde{\phi}^\tau_\mu(k)\exp(k^\sigma x_\sigma)$. Then, the exponential factors cancel and we get $\beta_{\mathcal{F}_l}(x_\mu) = x_\tau\phi^\tau_\mu(\partial) = \hy_\mu = \beta(\xi(x_\mu))$ as required.
\begin{gather*}
\alpha_{\mathcal{F}_r}(f)=\mu(\alpha_0\otimes\id)\exp(-x_\rho\otimes\partial^\rho)\exp(\hy_\sigma\otimes\partial^\sigma)(\triangleright f\otimes 1).
\end{gather*}
For $f = x_\mu$ calculate $\exp(\hy_\sigma\otimes\partial^\sigma)(\triangleright x_\mu\otimes 1) = (x_\mu\otimes 1)(\star\otimes\cdot)\exp(x_\sigma\otimes\partial^\sigma) = (\hx_\mu\otimes 1)(\triangleright \otimes\cdot)\exp(x_\sigma\otimes\partial^\sigma) = \exp(x_\sigma\otimes\partial^\sigma)(x_\tau\otimes\phi^\tau_\mu)$, hence $\alpha_{\mathcal{F}_r}(x_\mu) = \alpha_0(x_\tau)\phi^\tau_\mu=\hx_\mu=\alpha(\xi(x_\mu))$ as required.
\begin{gather*}
\beta_{\mathcal{F}_r}(f)=\mu(\beta_0\otimes\id)\exp(-\partial^\rho\otimes x_\rho)\exp(\partial^\sigma\otimes\hy_\sigma)(\triangleright f\otimes 1).
\end{gather*}
For $f = x_\mu$ only the three summands with partial derivatives up to the first order survive after applying $\triangleright x_\mu$. Thus
\begin{gather*}
\beta_{\mathcal{F}_r}(x_\mu) = \mu(\alpha_0\otimes\id)(x_\mu\otimes 1 - 1\otimes x_\mu+1\otimes\hy_\mu) = \hy_\mu = \beta(\xi(x_\mu))
\end{gather*}
as required.
\end{proof}

\begin{Corollary} $\tilde{\mathcal{F}}_l - \tilde{\mathcal{F}}_r \in I_{U(\gg)}$, $\mathcal{F}_l = \mathcal{F}_r$.
\end{Corollary}
\begin{proof} In the proof of Theorem~\ref{th:main} we have shown $\mu\tilde{\mathcal{F}}_l(\triangleright\otimes\triangleright)(f\otimes g)
= f\star g = \mu\tilde{\mathcal{F}}_r(\triangleright\otimes\triangleright)(f\otimes g)$ for all $f,g\in\hat{S}(\gg)$. Thus $\tilde{\mathcal{F}}_l -\tilde{\mathcal{F}}_r\in I_{U(\gg)}$ from the undeformed case of the Theorem~\ref{th:th4halg} on nondegeneracy. Therefore for the cosets we conclude $\mathcal{F}_l = \tilde{\mathcal{F}_l}+I_{U(\gg)} = \tilde{\mathcal{F}_r} + I_{U(\gg)} =\mathcal{F}_r$.
\end{proof}

\section{Twisting the antipode}\label{sec:antipode}

In this section, we discuss the twisting of the antipode. We first recall the definition of Hopf algebroids as bialgebroids with an antipode and then we recall from~\cite{halg} the antipode for $H_\gg$.

Several nonequivalent versions of the axioms for the antipode are used in the literature (see, e.g.,~\cite{bohmHbk,Bohm,lu,twosha}). In~\cite{halg} we checked for~$H_\gg$ the axioms of the symmetric Hopf algebroid, which involve both a left and a right bialgebroid. Thus a reasonable formalism for the twisting of symmetric Hopf algebroids is expected to provide a twist for the left and another twist for the right bialgebroid, these twists satisfying some compatibilities. Instead of taking this not so obvious path, we here work with a version of the axioms for the antipode involving only the left bialgebroid. We can do this because the antipode for $H_\gg$ is invertible both in the undeformed and deformed case. Namely, if the antipode is invertible, as proven by B\"ohm, the structure of a~symmetric Hopf $A$-algebroid~$H$ is equivalent to a left $A$-bialgebroid together with an antipode~$S$ which is an algebra antihomomorphism $H\to H$ having an inverse~$S^{-1}$ satisfying for all $h\in H$ the relations
\begin{gather*}
S\circ\beta = \alpha,\\
\big(S^{-1} h_{(2)}\big)_{(1)}\otimes_A\big(S^{-1} h_{(2)}\big)_{(2)}h_{(1)} = S^{-1} h\otimes_A 1_H,\\
\big(S h_{(1)}\big)_{(1)} h_{(2)}\otimes_A\big(S h_{(1)}\big)_{(2)} = 1_H\otimes_A S h.
\end{gather*}

If a bialgebroid twist $\mathcal{F} = \mathcal{F}^{(1)}\otimes_A\mathcal{F}^{(2)}$ on a Hopf algebroid $H$ has an inverse cocycle $\mathcal{F}^{-1}=\overline{\mathcal{F}}^{(1)}\otimes_{A_\star}\overline{\mathcal{F}}^{(2)}$ in the sense of Remark~\ref{rem:cocyccoass}, then $V_{\mathcal{F}} = \big(S\mathcal{F}^{(1)}\big)\mathcal{F}^{(2)}$ is a well-defined element in~$H$. If~$H$ is a Hopf algebra then a standard calculation shows that $\overline{\mathcal{F}}^{(1)}\big(S\overline{\mathcal{F}}^{(2)}\big)$ is the two-sided inverse of~$V$ with respect to the multiplication in $H$. The calculation does not extend to Hopf algebroids, namely not only that $\overline{\mathcal{F}}^{(1)}\big(S\overline{\mathcal{F}}^{(2)}\big)$ is not the inverse of $V$, but worse, $\overline{\mathcal{F}}^{(1)}\big(S\overline{\mathcal{F}}^{(2)}\big)$ is not a well defined expression because $\mu(\id\otimes S)I_{A_\star}\neq 0$. We do not know if the inverse of $V_{\mathcal{F}}$ exists in general.

\begin{Proposition}\label{prop:VFSF}
Suppose that $V_{\mathcal{F}} = \big(S\mathcal{F}^{(1)}\big)\mathcal{F}^{(2)}$ has an inverse $V^{-1}_{\mathcal{F}}$ in $H$. Then define
\begin{gather}\label{eq:SF}
S_{\mathcal{F}} h = V^{-1}_{\mathcal{F}} (S h) V_{\mathcal{F}} = V^{-1}_{\mathcal{F}} (S h) \big(S\mathcal{F}^{(1)}\big)\mathcal{F}^{(2)}.
\end{gather}
The formula $h\mapsto S_{\mathcal{F}} h$ then defines an antihomomorphism of algebras $S_{\mathcal{F}}\colon H\to H$ and
\begin{gather}\label{eq:VmF}
V^{-1}_{\mathcal{F}}= \big(S_{\mathcal{F}}\overline{\mathcal{F}}^{(1)}\big)\overline{\mathcal{F}}^{(2)},
\end{gather}
where the right-hand side is well defined and in particular $\mu(S_{\mathcal{F}}\otimes\id) I_{A_\star} = 0$.
\end{Proposition}
\begin{proof}
$S$ is an antihomomorphism hence it is clear that $S_{\mathcal{F}}$ given by~(\ref{eq:SF}) is an antihomomorphism as well. It follows that $\big(S\overline{\mathcal{F}}^{(1)}\big)\big(S\mathcal{F}^{(1)'}\big)\mathcal{F}^{(2)'}\overline{\mathcal{F}}^{(2)} = 1$. Here the primed Sweedler indi\-ces~$(1)'$,~$(2)'$ refer to another copy of $\mathcal{F}$
\begin{align*}
\big(S_{\mathcal{F}}\overline{\mathcal{F}}^{(1)}\big)\overline{\mathcal{F}}^{(2)}\big(S\mathcal{F}^{(1)}\big)\mathcal{F}^{(2)} & =
V^{-1}_{\mathcal{F}}\big(S\overline{\mathcal{F}}^{(1)}\big)V_{\mathcal{F}}\overline{\mathcal{F}}^{(2)}\big(S\mathcal{F}^{(1)}\big)\mathcal{F}^{(2)}
\\ & = V^{-1}_{\mathcal{F}}\big(S\overline{\mathcal{F}}^{(1)}\big)\big(S\mathcal{F}^{(1)'}\big)\mathcal{F}^{(2)'}\overline{\mathcal{F}}^{(2)}
\big(S\mathcal{F}^{(1)}\big)\mathcal{F}^{(2)}
\\ & = V^{-1}_{\mathcal{F}}(S\mathcal{F}^{(1)})\mathcal{F}^{(2)}
\\ & = 1,
\\
\big(S\mathcal{F}^{(1)}\big)\mathcal{F}^{(2)}\big(S_{\mathcal{F}}\overline{\mathcal{F}}^{(1)}\big)\overline{\mathcal{F}}^{(2)} & = \big(S\mathcal{F}^{(1)}\big)\mathcal{F}^{(2)}V^{-1}_{\mathcal{F}}\big(S\overline{\mathcal{F}}^{(1)}\big)V_{\mathcal{F}}\overline{\mathcal{F}}^{(2)}
\\ & = \big(S\mathcal{F}^{(1)}\big)\mathcal{F}^{(2)}V^{-1}_{\mathcal{F}}\big(S\overline{\mathcal{F}}^{(1)}\big)
\big(S\mathcal{F}^{(1)'}\big)\mathcal{F}^{(2)'}\overline{\mathcal{F}}^{(2)}
\\ & = \big(S\mathcal{F}^{(1)}\big)\mathcal{F}^{(2)}V^{-1}_{\mathcal{F}}
\\ & = 1.
\end{align*}

Regarding that the two-sided inverse in an associative algebra is unique, we conclude~(\ref{eq:VmF}) with the right-hand side in~(\ref{eq:VmF}) well defined.
\end{proof}

One would like to conclude that $S_{\mathcal{F}}$ is an antipode for the twisted bialgebroid. The standard proofs for the Hopf algebras do not seem to generalize in straightforward manner.

However, in our case, for $H_\gg$, we know the deformed antipode $S$, and one can say a bit more. For $H_\gg\cong\hat{A}_n$, the antipode $S_0$ for the undeformed coproduct is a continuous antihomomorphism $S_0\colon H_\gg\to H_\gg$ given on the generators $x_\mu,\partial^\nu$ of the dense subalgebra $A_n$ by $S_0(x_\mu)=x_\mu$, $S_0(\partial^\nu) = -\partial^\nu$. Similarly, the antipode $S\colon H_\gg\to H_\gg$ for the deformed coproduct is determined by the formulas
\begin{gather*}
S\hat{y}_\mu = \hat{x}_\mu,\qquad S\partial^\nu = S_0\partial^\nu = -\partial^\nu.
\end{gather*}
 Therefore using $\hat{x}_\mu = x_\rho \phi(\partial)^\rho_\mu$ and $\hat{y}_\mu = x_\rho \phi(-\partial)^\rho_\mu$
\begin{gather*}
S_0 \hat{y}_\mu = S_0(x_\rho \phi(-\partial)^\rho_\mu) = \phi(\partial)^\rho_\mu x_\rho = x_\rho\phi(\partial)^\rho_\mu + \partial_\rho\phi(\partial)^\rho_\mu,
\end{gather*}
where we denoted $\partial_\rho = \frac{\partial}{\partial(\partial^\rho)}$.

We seek for $V$ such that $S(h) = V^{-1}S_0(h)V$, in parallel to equation~(\ref{eq:SF}) defining $S_{\mathcal{F}}$ in terms of $V_{\mathcal{F}}$. Regarding that $S(\partial^\mu) = S_0(\partial^\mu)= -\partial^\mu$, this forces that $[\partial^\mu,V] = 0$, hence $V = V(\partial^1,\dots,\partial^n)\in\hat{S}(\gg^*)$. Moreover,
\begin{gather}
S \hat{y}_\mu = x_\rho\phi^\rho_\mu = V^{-1}\big(x_\rho\phi(\partial)^\rho_\mu + \partial_\rho\phi(\partial)^\rho_\mu\big)V = V^{-1}x_\rho V\phi^\rho_\mu + \partial_\rho\phi^\rho_\mu,\nonumber\\
x_\mu - \big(\partial_\rho \phi^\rho_\gamma\big)\phi^{-1\gamma}_\mu = V^{-1} x_\mu V,\nonumber\\
V^{-1} [V,x_\mu] = (\partial_\rho\phi^\rho_\gamma)\phi^{-1\gamma}_\mu,\nonumber\\
\partial_\mu{\ln|V|} = V^{-1}\partial_\mu V = \phi^{-1\gamma}_\mu\partial_\rho\phi^\rho_\gamma,\nonumber\\
R := \ln|V|,\nonumber\\
\partial_\mu R = \phi^{-1\gamma}_\mu\partial_\rho\phi^\rho_\gamma.\label{eq:R1}
\end{gather}
Setting $F_\mu = \phi^{-1\gamma}_\mu\partial_\rho\phi^\rho_\gamma$, we rewrite~(\ref{eq:R1}) as the system of formal differential equations
\begin{gather}\label{eq:R2}
\partial_\mu R = F_\mu
\end{gather}
for an unknown formal series $R$. In our case, $F_\mu$ are analytic functions, hence the solution exists if the integrability condition $\partial_\nu F_\mu= \partial_\mu F_\nu$ holds, which boils down to
\begin{gather}\label{eq:integrability}
\partial_\nu\big(\phi^{-1\gamma}_\mu\partial_\rho\phi^\rho_\gamma\big) = \partial_\mu\big(\phi^{-1\gamma}_\nu\partial_\rho\phi^\rho_\gamma\big).
\end{gather}
Any solution for $V = \exp(\ln|V|) = \exp(R)$ is clearly invertible with an inverse $V^{-1} = \exp(-R)$. The constant term $V(0)$ of $V$ which is viewed as a formal power series is nonzero, hence $V(0)$ is also invertible. Therefore we can write $V = V(0)\cdot V_1$ where $V_1 = V/V(0)$. It is clear that $V_1$ is also a solution and the identity $V_1(0) = 1$ holds. We conclude that without loss of generality (by passing from $V$ to $V_1$) we may assume that $V(0) = 1$ and such a solution for $V$ is unique. The condition $V(0)=1$ is equivalent to the boundary condition $R(0) = 0$ for the ``potential'' $R$ when solving the exact first order differential equation~(\ref{eq:R2}). This boundary condition guarantees the uniqueness of the solution for $R$. If, instead of the abstract algebra elements $\partial^\nu$, we introduce the real variables $\xi^\nu$, we can write the solution formally
\begin{gather*}
R\big(\zeta^1,\dots,\zeta^n\big) = \int_\Gamma \sum F_\mu\big(\xi^1,\dots,\xi^n\big) {\rm d} \xi^\mu,
\end{gather*}
where the line integral is along any path $\Gamma$ from $(0,\dots,0)$ to $(\zeta^1,\dots,\zeta^n)$. Due to the integrability~(\ref{eq:integrability}), the line integral does not depend on the path chosen.

We have some evidence that $S_{\mathcal{F}} = S$ or equivalently that $V_{\mathcal{F}}$ satisfies the equations for $V$. First of all, Proposition~\ref{prop:Fpartial} below says that the element $V_{\mathcal{F}}\in\hat{A}_n$ is in fact the formal power series in $\partial^1,\dots,\partial^n$ (that is $x_1,\dots, x_n$ do not appear) making sense of the equations for $V$. $V_{\mathcal{F}}$~clearly satisfies the required boundary condition $V_{\mathcal{F}}(0) = 1$. It follows that $V_{\mathcal{F}}$ is invertible, as required in Proposition~\ref{prop:VFSF}, particularly in equation~(\ref{eq:SF}) defining $S_{\mathcal{F}}$. Using the Baker--Hausdorff formula for computing $R_\mathcal{F} = \ln|V_{\mathcal{F}}|$ directly from the definition $V_{\mathcal{F}}=\mu(S_0\otimes\id)\mathcal{F}$, we have perturbatively checked up to the third order in $\partial$-s that $R_{\mathcal{F}}$ satisfies the system~(\ref{eq:R2}) for $R$.

\begin{Proposition}\label{prop:Fpartial}
For $\mathcal{F}$ being $\mathcal{F}_c = \mathcal{F}_l = \mathcal{F}_r$, the element $V_{\mathcal{F}}\in H_\gg$ is a formal power series in $\partial^1,\dots,\partial^n$. As a corollary, $S_{\mathcal{F}}(\partial^\nu) = S_0 \partial^\nu = S \partial^\nu$.
\end{Proposition}
The corollary part of the proposition follows by the defining formula~(\ref{eq:SF}) for $S_{\mathcal{F}}$ recalling that $\partial^1,\dots,\partial^n$ mutually commute. The first statement in Proposition~\ref{prop:Fpartial} is the content of parts~(v) and~(vi) of the following lemma restated.
\begin{Lemma}\quad
\begin{enumerate}\itemsep=0pt
\item[$(i)$] $(S_0\otimes\id)\tilde{\mathcal{F}}_l\in H_\gg\otimes H_\gg$ can be written as $\exp(W)$ where $W$ is a formal sum of summands which are up to a rational factor equal to $\partial^{\alpha_1}\cdots\partial^{\alpha_s}\otimes L(w_{\alpha_1},\dots,w_{\alpha_s})$, where $s\geq 1$ and $w_{\alpha_i}$ is either $x_i$ or $\hat{x}_i$ $($the choice depending on $i$ and on the summand$)$ and~$L$ is a~Lie monomial $(L$ differing from a summand to summand$)$, more precisely an iterated Lie bracket involving each of the variables precisely once.

\item[$(ii)$] For every of the summands in $(i)$ it holds that $L(w_{\alpha_1},\dots,w_{\alpha_s}) = \sum_\rho x_\rho\zeta^{\rho}_{\alpha_1,\dots,\alpha_s}$ where $\zeta^{\rho}_{\alpha_1,\dots,\alpha_s}\in\hat{S}(\gg^*)$.

\item[$(iii)$] $(\phi^\mu_\nu-\delta^\mu_\nu)\partial^\nu=0$, therefore $(\hat{x}_\nu-x_\nu)\partial^\nu = 0$. The summands in $W$ with $s = 1$ add up to $\partial^\nu\otimes(\hat{x}_\nu-x_\nu)$.

\item[$(iv)$] For every of the summands in $(ii)$ with $s\geq 2$, $\partial^{\alpha_1}\cdots\partial^{\alpha_s}\cdot\zeta^{\rho}_{\alpha_1,\dots,\alpha_s} = 0$.

\item[$(v)$] $V_{\mathcal{F}_l} = \mu(S_0\otimes\id)\mathcal{F}_l$ belongs to $\hat{S}(\gg^*)$.

\item[$(vi)$] $V_{\mathcal{F}_c} = \mu(S_0\otimes\id)\mathcal{F}_c$ $($see~\eqref{eq:Fc}$)$ belongs to $\hat{S}(\gg^*)$.
\end{enumerate}
\end{Lemma}

\begin{proof}(i) The statement is a straightforward application of the Baker--Hausdorff series to $\tilde{\mathcal{F}}_l=\exp(-\partial^\rho\otimes x_\rho)\exp(\partial^\sigma\otimes\hx_\sigma)$ and using $S_0(\partial^\mu)=-\partial^\mu$.

(ii) Easily follows by induction, computing in~$\hat{A}_n$.

(iii) Recall~(\ref{eq:phi}) that $\phi = \sum\limits_{N=0}^\infty \frac{(-1)^N B_N}{N!} \CC^N$. If $N\geq 1$ then $(\mathcal{C}^N)^\mu_\nu\partial^\nu = (\mathcal{C}^{N-1})^\mu_\rho C^\rho_{\nu\sigma}\partial^\nu\partial^\sigma = 0$ as it involves a contraction of an antisymmetric tensor $C^\rho_{\mu\nu}$ in~$\mu$, $\nu$ with a symmetric tensor~$\partial^\mu\partial^\nu$.

(iv) Along with the statement to prove, we also claim that $\zeta$ is an expression which is a~contraction (always one lower and one upper index are contracted) of the tensors of the form $\mathcal{C}^\gamma_\delta$ and $C^\gamma_{\delta\lambda}$ where the contracted form is connected and the only external indices which remain are the labels of~$\zeta$. In particular, each~$\alpha_i$ is either within a factor of the form $\mathcal{C}^\gamma_{\alpha_i} = C^\gamma_{\alpha_i\lambda}\partial^\lambda$, which contracted with $\partial^{\alpha_i}$ is clearly zero (because $C^\gamma_{\alpha_i\lambda}$ is antisymmetric in lower and $\partial^{\alpha_i}\partial^\lambda$ in upper labels), or~$\alpha_i$ appears in a factor of the form $C^\gamma_{\alpha_i\alpha_j}$ which is zero when contracted with~$\partial^{\alpha_i}\partial^{\alpha_j}$. This description is an easy observation for $s = 2$, and then we proceed by induction on~$s$. Note that for any~$\zeta$,~$\zeta'$ tensors whose components are formal series in $\partial^1,\dots,\partial^n$ there are new tensors~$\zeta''$ defined by the commutators $\big[\sum_\rho x_\rho\zeta^\rho_{\alpha_1,\dots,\alpha_s},\sum_\sigma x_\sigma(\zeta')^\sigma_{\beta_1,\dots,\beta_r}\big] = \sum_\tau x_\tau(\zeta'')^\tau_{\alpha_1,\dots,\alpha_s\beta_1,\dots,\beta_r}$. Explicitly,{\samepage
\begin{gather*}
(\zeta'')^\tau_{\alpha_1,\dots,\alpha_s\beta_1,\dots,\beta_r} =
\frac{\partial}{\partial(\partial^\lambda)}\big(\zeta^\tau_{\alpha_1,\dots,\alpha_s}\big) \cdot(\zeta')^\lambda_{\beta_1,\dots,\beta_r}-\frac{\partial}{\partial(\partial^\lambda)}\big((\zeta')^\lambda_{\beta_1,\dots,\beta_r}\big) \cdot\zeta^\tau_{\alpha_1,\dots,\alpha_s}.
\end{gather*}
For the induction step one checks that if~(iv) holds for $\zeta$ and~$\zeta'$ then it holds also for~$\zeta''$.}

(v) This is somewhat tricky point because the multiplication does not commute with eva\-luating the exponential series. $\mu(S_0\otimes\id)\mathcal{F}_l$ will be a sum of terms which are proportional to products of the summands from part~(i), that is of the form
\begin{gather*}\big(\partial^{\alpha_1}\cdots\partial^{\alpha_s}\big)\big(\partial^{\beta_1}\cdots\partial^{\beta_r}\big)\cdots \big(x_{\rho}\zeta^\rho_{\alpha_1,\dots,\alpha_s}\big)\big(x_{\sigma}\zeta^\sigma_{\beta_1,\dots,\beta_r}\big)\cdots.\end{gather*}
Moving $x$-s to the left by commuting we bring this expression to the normally ordered form. By~(iv), a nonzero contribution can possibly happen only from those terms in which each bracket of partials on the left has lost at least one power by yielding a~Kronecker delta when commuted with some $x$. There is only one~$x$ per bracket on the right-hand side; because the $\partial$-groups on the left and the $x\zeta$-groups on the right match, that means that all $x$-s disappear in the normally ordered form.

(vi) While this is trivially equivalent to the statement in~(v), there is an independent but similar proof for $\mathcal{F}_c$, using the defining formula~(\ref{eq:Fc}) and the form of the tensors appearing in developing~$\tilde{\mathcal{F}}_c$, as described in~\cite{scopr}. In particular, one needs to observe that for every \smash{$s\geq 1$} the expression $\partial^{\mu_1}\cdots\partial^{\mu_s}\cdot
[\cdots[\partial^\lambda,\hat{x}_{\mu_1}],\ldots,\hat{x}_{\mu_s}]$, which appears as a summand in $\Delta_{\hat{S}(\gg^*)}\partial^{\lambda}$, va\-ni\-shes. To this end, one uses the description in~\cite{scopr} of the tensors which appear in developing
$[\cdots[\partial^\lambda,\hat{x}_{\mu_1}],\dots,\hat{x}_{\mu_s}]$.
\end{proof}

\subsection*{Acknowledgements} S.M.~has been supported by Croatian Science Foundation under the Project no.~IP-2014-09-9582 and the H2020 Twinning project no.~692194 ``RBI-T-WINNING''. Z.\v{S}.~has been partly supported by grant no.~18-00496S of the Czech Science Foundation. We thank A.~Borowiec for his remarks on the paper and M.~Stoji\'c for remarks on Sections~\ref{section1} and~\ref{sec:bialgebroids}. We thank the referees for bringing to our attention numerous constructive suggestions, which helped extending and improving the article significantly.

\pdfbookmark[1]{References}{ref}
\LastPageEnding

\end{document}